\documentclass[12pt]{amsart}
\usepackage{amsthm,amsmath,amsxtra,amscd,amssymb,xspace,esint}
\usepackage[cp1251]{inputenc}
\usepackage[T2A]{fontenc}

\numberwithin{equation}{section}

\def\beq{\begin{equation}}
\def\eeq{\end{equation}}
\def\p{\partial}
\def\G{\Gamma}
\def\g{\gamma}
\def\s{\sigma}

\def\e{\varepsilon}
\def\L{{\cal L}}

\def\a{\alpha}
\def\b{\beta}
\def\l{\lambda}

\def\A{{\mathcal A}}

\def\D{{\mathcal D}}

\def\L{{\mathcal L}}

\def\M{{\mathcal M}}

\def\P{{\cal P}}

\def\dim{{\rm dim}}
\def\res{{\rm res}}

\def\wt{\widetilde}
\def\wh{\widehat}
\def\bH{\mathbb H}
\def\bC{\mathbb C}
\def\J{\mathcal J}
\def\bZ{\mathbb Z}
\def\P{\mathcal P}

\def \matrix #1 {\left(\begin{array}{cc} #1 \end{array}\right)}
\newtheorem{theorem}{Theorem}[section]

\newtheorem{corollary}{Corollary}[section]
\newtheorem{lemma}{Lemma}[section]
\newtheorem{remark}{Remark}

\begin{document}

\title{Abelian pole systems and Riemann-Schottky type systems}

\author{Igor Krichever}
\address{Columbia University, New York, USA, and Skolkovo Institute for Science and Technology, and National Research University Higher School of Economics, Moscow, Russia}
\email{krichev@math.columbia.edu}

\begin{abstract} In this survey of  works on a characterization of Jacobians and Prym varieties among indecomposable principally polarized abelian varieties via the soliton theory we focus on a certain circle of ideas and methods which show that the characterization of Jacobians as ppav whose Kummer variety admits a trisecant line and the Pryms as ppav whose Kummer variety admits a pair of symmetric quadrisecants can be seen as an abelian version of pole systems arising in the theory of elliptic solutions to the basic soliton hierarchies. We present also recent results in this direction on the characterization of Jacobians of curves with involution, which were motivated by the theory of two-dimensional integrable hierarchies with symmetries.

\end{abstract}

\maketitle

\section{Introduction}

\medskip

Novikov's conjecture on the Riemann-Schottky problem: \emph{the Jacobians of smooth algebraic curves are precisely those indecomposable principally polarized abelian varieties (ppavs) whose theta-functions provide solutions to the Kadomtsev-Petviashvili (KP) equation}, was the first evidence of nowadays well-established fact: connections between the algebraic geometry and the modern theory of integrable systems is beneficial for both sides. Novikov' conjecture was proved by T.~Shiota in \cite{shiota}.

\smallskip

The first goal of this paper is to present the strongest known characterization of a Jacobian variety in this direction: \emph{an indecomposable ppav $X$ is the Jacobian of a curve if and only if its Kummer variety $K(X)$ has a trisecant line\/} proved in \cite{kr-schot,kr-tri}.
This characterization is called \emph{ Welters' (trisecant) conjecture\/} after the work of Welters \cite{wel1} which was motivated by Novikov's conjecture and Gunning's celebrated theorem \cite{gun1}. The approach to its solution, proposed in \cite{kr-schot}, is general enough to be applicable to a variety of Riemann-Schottky-type problems. In \cite{kr-quad,kr-prym} it was used for a characterization of principally polarized Prym varieties. The latter problem is almost as old and famous as the Riemann-Schottky problem but is much harder.

\smallskip
Our second goal is to present recent results on characterization of Jacobians of curves with involution.
The curves with involution naturally appears as a part of algebraic-geometrical data defining solutions to integrable system with symmetries. Numerous examples of such systems  include the Kadomtsev-Petviashvili hierarchies of type B and C (BKP and CKP hierarhies, respectively) introduced in  \cite{DJKM81,DJKM83}  and  the Novikov-Veselov hierarchy introduced in \cite{nv1,nv2}.

The existence of an involution of a curve is central in proving that the constructed solutions have the necessary symmetry. The solutions corresponding to the same curve are usually parameterized by points of its Prym variety. In other words the existence of involution plus some extra constraints on the divisor of the Baker-Akhiezer function are \emph{ sufficient} conditions ensuring required symmetry. The problem of proving that these conditions are \emph{necessary} for \emph{two-dimensional} integrable hierarchies is much harder and that is the  problem solved in \cite{kr-inv}.

The third and to some extend our primary objective is to take this opportunity to elaborate on motivations underlining the proposed solution of the Riemann-Schottky type problems and to introduce a certain circle of ideas and methods.

May be the most important among them is a mysterious \emph{ generating  property} of two-dimensional linear differential, differential-functional, difference-functional equations. In fact, we will discuss two sources of generating properties. One of them is \emph{local}, and it concerns equations with \emph{ meromorphic} coefficients in one of the variables that have \emph{meromorphic solutions}. The other is \emph{global} and concerns equations with elliptic coefficients that have solutions that are meromorphic sections of a line bundle over an elliptic curve.

\medskip

Three main examples are:

\smallskip
$(i)$ the differential equation
\beq\label{intCM}
(\p_t-\p_x^2+u(x,t))\psi(x,t)=0, \quad u=-2\p_x^2\tau(x,t),
\eeq

$(ii)$ the differential-functional equation
\beq\label{intRS}
\p_t \psi(x,t)=\psi(x+1,t)+w(x,t)\psi(x,t), \quad w(x,t)=\p_t\ln \left(\frac{\tau(x+1,t)}{\tau(x,t)}\right),
\eeq

$(iii)$ the difference-functional equation
\beq\label{intBA}
\psi_{n+1}(x)=\psi_n(x+1)-v_n(x)\psi_n(x),\quad  v_n(x)=\frac{\tau_n(x)\tau_{n+1}(x+1)}{\tau_{n}(x+1)\tau_{n+1}(x)}
\eeq
with unknown functions  $\psi_n(x)\,,{n\in\bZ}$.

Each of these equations (after change of notations for independent variables) is one out of two auxiliary linear problems for the three fundamental equations of the theory of integrable systems: the Kadomtsev-Petviashvili (KP) equation
\beq\label{kp}
{3}u_{yy}=\left(4u_t-6uu_x+u_{xxx}\right)_x\,,
\eeq
the $2D$ Toda equation
\beq\label{2DT}
\p_\xi\p_\eta\varphi_n=e^{\varphi_n-\varphi_n}-e^{\varphi_{n}-\varphi_{n+1}}, \quad \varphi_n=\varphi(x=n,\xi,\eta)\,,
\eeq
and the Bilinear Discrete Hirota equation (BDHE)
\beq\begin{aligned}\label{BDHE}
\tau _n (l+1,m)\tau _n (l,m+1)- \tau _n (l,m)\tau _n (l+1,m+1)+ \\
+\tau _{n+1} (l+1,m)\tau _{n-1} (l,m+1)=0\,,
\end{aligned}
\eeq

\noindent
respectively.

At the first glance all three nonlinear equation: the KP equation, the 2D Toda equation, and the BDHE equation, look very different from each other.  But in the theory of integrable systems it is well-known that these
fundamental soliton equations are in intimate relation: the KP equation is as a continuous limit of the BDHE, and the 2D Toda equation can be obtained in an intermediate step.

\medskip

Assume that in the first two cases $\tau(x,t)$ is an \emph{entire} function of the variable $x$ and a (local) smooth function of the variable $t$, and in the third case $\tau_n(x)$ is a sequence of entire functions of $x$. It turns out that under some generality assumption for each of the above linear equations the answer to the question  \emph{when it has a meromorphic in $x$ solution} is given in terms of equations describing the evolution of  zeros of $\tau$ in the second variable.

To give an idea of these equations and why I called the very existence of them mysterious, as an  instructive example, consider the equation \eqref{intCM}.

Let $\psi$ be a meromorphic solution of \eqref{intCM} with $u=-2\p_x^2 \ln \tau(x,t) $, where $\tau$ is a entire  function of $x$ and smooth function of $t$ in some neighborhood of $t=0$. The generality assumption is that generic zeros of $\tau$ are simple. Consider the Laurent expansions of $\psi$ and $u$ in the neighborhood of a simple zero, $\tau(q(t),t)=0, \p_x\tau(q(t),t)\neq 0$:
\begin{equation}
\begin{aligned}\label{ue}
u&=\frac 2{(x-q)^2}+v+w(x-q)+\ldots;\\
\psi&=\frac{\a}{x-q}+\b+\g(x-q)+\delta(x-q)^2+\ldots.\\
\end{aligned}
\end{equation}
(The coefficients in these expansions $v,w,\ldots; \a,\b,\ldots$ are smooth functions of the variable $t$).
Substitution of (\ref{ue}) in (\ref{intCM}) gives an \emph{infinite} system of equations. The first three of them
are
\begin{equation}\begin{aligned}\label{eq1}
\a\dot q+2\b=0;\\
\dot \a+\a v+2\g=0;\\
\dot\b+v\b-\g\dot q+\a w=0.\\
\end{aligned}
\end{equation}

Taking the $t$-derivative of the first equation and using two others we get the equation

\beq\label{cm50}
\ddot q=2w,\
\eeq
derived  first in \cite{flex}.

\smallskip

We would like to emphasize once again that there is no reason for the fact that the system
\eqref{eq1}  can be reduced to equations for the potential $ u $, only. Even more unexpected
for the author was that, as we will see later, the existence of \emph{one} meromorphic solution of the equation \eqref {intCM} is sufficient for the existence of a \emph{one-parameter} family of meromorphic solutions.
\smallskip

Formally, if we represent $\tau$ as an infinite product,
\beq\label{roots}
\tau(x,t)=c(t)\prod_i(x-q_i(t)),
\eeq
then equation (\ref{cm50}) can be written as the infinite system of equations
\beq\label{cm6}
\ddot q_i=-4\sum_{j\neq i} \frac{1}{(q_i-q_j)^3}\,.
\eeq
Equations (\ref{cm6}) are purely formal because, even if $\tau$ has simple zeros at $t=0$,
then in the general case there is no nontrivial interval in $t$ where the zeros remain simple.
One of the reason to present \eqref{cm6} is that it shows that when $\tau$ is a rational, trigonometric or
elliptic polynomial  then equations \eqref{cm6}  coincide with the equations of
motion for the rational, trigonometrical or elliptic Calogero-Moser (CM) systems, respectively.

In a similar way one can get that the existence of a meromorphic solution for equations \eqref{intRS} and \eqref{intBA} gives equations on zeros of $\tau$ which in the case when $\tau$ is an elliptic polynomial in $x$ turned out to  be equations of motion of the elliptic Ruijsenaars-Schneider (RS) model and nested Bethe ansatz equations, respectively.

\medskip
Recall, that the elliptic CM system with $k$ particles is a Hamiltonian system
with coordinates $q=(q_1,\dots,q_k)$, momentums $p=(p_1,\dots,p_k)$,  the canonical Poisson brackets $\{q_i,p_j\}=\delta_{ij}$, and  the Hamiltonian
\beq\label{HCM}
H=\frac 1 2 \sum_{i=1}^k p_i^2 +\sum_{i\neq j} \wp(q_i-q_j)\,.
\eeq
The corresponding equations of motion
admit the  Lax representation $\dot L=[M,L]$ with
\beq\label{LCM}
L_{ij}\,=\,p_i\delta_{ij}\,+\,2\,(1-\delta_{ij}) \Phi(q_i-q_j,z)
\eeq
where
\beq\label{PhiCM}
\Phi(x,z)=\frac{\s(z-x)}{\s(z)\s(z)}e^{x\zeta(z)}
\eeq
and $\s,\zeta,\wp$ are classical Weierstrass functions.

\medskip
The elliptic RS system is a Hamiltonian system with coordinates $q=(q_1,\dots,q_k)$,
momentums $p=(p_1,\dots,p_k)$,  the canonical Poisson brackets $\{q_i,p_j\}=\delta_{ij}$, and  the Hamiltonian
\beq\label{RSHam}
H=\sum_{i=1}^k f_i
\eeq
where
\beq\label{gadef}
f_i:=e^{p_i}\prod_{j\neq i} \left(\frac{\s(q_i-q_j-1)\s(q_i-q_j+1)}{\s(q_i-q_j)^2}\right)^{1/2}.
\eeq
It is a completely integrable Hamiltonian system, whose equations of motion
admit the Lax representation $\dot L=[M,L]$, where
\beq\label{dLR}
L_{ij}={f_i}\,\Phi{(q_i-q_j-1,z)},\
\qquad i,j=1,\dots, k,
\eeq

The elliptic nested Bethe ansatz equations
are a system of algebraic equations
\beq
\prod_{j}\frac{\s(q_i^n-q_j^{n+1})\s(q_i^n-1-q_j^n)\s(q_i^n-q_j^{n-1}+1)}
{\s(q_i^n-q_j^{n-1})\s(q_i^{n+1}-q_j^n)\s(q_i^n-q_j^{n+1}-1)}=-1
\label{bet}
\eeq
for $k$ unknown functions $q_i=\{q_i^n\},\ i=1,\ldots, k$, of a discrete
time variable $n\in \mathbb Z$.

The above systems are usually called elliptic \emph{pole} systems, since they describe the dependence of the poles of the elliptic solutions of the KP, $ 2D $ Toda and BDHE equations, respectively.
A correspondence between finite-dimensional integrable systems and the pole
systems of various soliton equations was considered in \cite{krbab,zab,krwz,kr3}. In \cite{akv} it was generalized to the  case of \emph{field analogues} of CM type systems.

\medskip
The most general form of the function $ \tau $, known to the author so far, for which the equations for its zeros are not formal, is the case of \emph{abelian} functions, that is, when $ \tau $ has the form
\beq\label{tau1}
\tau=\tau(Ux+z,t)\,,
\eeq
where $x, t\in\mathbb C$ and $z\in \mathbb C^n$ are independent variables,
$0\ne U\in\mathbb Cn$, and for all $t$ the function
$\tau(\cdot,t)$ is a holomorphic section of a line bundle $\L=\L(t)$ on an
abelian variety $X=\bC^n/\Lambda$, i.e., for all $\l\in\Lambda$
it satisfies the monodromy relations
\begin{equation}\label{taumon}
\tau(z+\l,t)=e^{a_\l\cdot z+b_\l}\tau(z,t),
\end{equation}
for some $a_\l\in\mathbb C^n$, $b_\l=b_\l(y,t)\in\mathbb C$.

It is tempting to call them \emph{abelian} CM, RS and nested Bethe ansatz equations. As  we shall see below they are central  for  the proof of three particular cases of the Welters conjecture.

\section{Riemann-Schottky problem}

Let $\bH_g:=\{B\in M_g(\bC)\mid {}^tB=B,\ {Im} (B)>0\}$ be the Siegel upper half space.
For $B\in\bH_g$ let $\Lambda:=\Lambda_B:=\bZ^g+B\bZ^g$ and $X:=X_B:=\bC^g/\Lambda_B$.
Riemann's theta function
\beq\label{teta1}
\theta(z):=\theta(z,B):=\sum_{m\in\bZ^g}e^{2\pi i(m,z)+\pi i(m,Bm)},\quad
(m,z)=m_1z_1+\cdots+m_gz_g,
\eeq
is holomorphic and $\Lambda$-quasi\-periodic in $z\in\bC^g$.

The factor space $\bH_g/Sp(2g,\bZ)\simeq\A_g$ is the moduli space of $g$-dimensional ppavs.
A ppav $(X,[\Theta])\in\A_g$ is said to be \emph{indecomposable\/}
if the zero-divisor $\Theta$  of $\theta$ is irreducible.

Let $\M_g$ be the moduli space of nonsingular curves of genus $g$,
and let $J\colon\M_g\to\A_g$ be the Jacobi map defined by the composition of maps $\M_g\to \bH_g\to \A_g$. The first one requires a choice of  a symplectic basis $a_i$, $b_i$ ($i=1,\dots,g$) of
$H_1(\G,\bZ)$ which defines a basis $\omega_1$, \dots, $\omega_g$ of the space of
holomorphic 1-forms on $\G$ such that $\int_{a_i}\omega_j=\delta_{ij}$, and then
the \emph{period matrix\/} and the \emph{Jacobian variety\/} of $\G$ by
\beq\notag
B:=\Big(\int_{b_i}\omega_j\Big)\in\bH_g \quad\hbox{and}\quad
J(\G):=(X_B,[\Theta_B])\in\A_g\,,
\eeq
respectively.

$J(\G)$ is indecomposable and the Jacobi map $J$ is injective (Torelli's theorem).
The \emph{(Riemann-)Schottky problem\/} is the problem of characterizing the
Jacobi locus $\J_g:=J(\M_g)$ or its closure $\overline{\J_g}$ in $\A_g$.
For $g=2$,~$3$ the dimensions of $\M_g$ and $\A_g$ coincide,
and hence $\overline{\J_g}=\A_g$ by Torelli's theorem.
Since $\J_4$ is of codimension $1$ in $\A_4$, the case $g=4$ is the first nontrivial case of the Riemann-Schottky problem.

A nontrivial relation for the Thetanullwerte of a curve of genus~$4$
was obtained by F.~Schottky \cite{schottky} in 1888, giving a modular form which vanishes on
$\J_4$, and hence at least a \emph{local\/} solution of the Riemann-Schottky
problem in $g=4$, i.e., $\overline{\J_4}$ is \emph{an irreducible component
of\/} the zero locus $\mathcal S_4$ of the Schottky relation.
The irreducibility of $\mathcal S_4$ was proved by Igusa \cite{igusa}
in 1981, establishing $\overline{\J_4}=\mathcal S_4$, an effective answer
to the Riemann-Schottky problem in genus 4.

Generalization of the Schottky relation to a curve of higher genus, the so-called Schottky-Jung relations, formulated as a conjecture by Schottky and Jung \cite{schot-jung} in 1909,
were proved by Farkas-Rauch \cite{far} in 1970. Later, van~Geemen \cite{geemen} proved that the Schottky-Jung relations give a local solution of the Riemann-Schottky problem. They do not give a global solution when $g>4$, since the variety they define has extra components already for $g=5$ (Donagi \cite{donagi2}).

\medskip

Over more than 120 year-long history of the Riemann-Schottky problem,
quite a few geometric characterizations of the Jacobians have been obtained.
None of them provides an explicit system of equations for the image of the Jacobian locus
in the projective space under the level two theta imbedding.

Following Mumford's review with a remark on Fay's trisecant formula
\cite{mum:c+j}, and the advent of algebraic geometrical integration scheme in the soliton theory  \cite{kr1,kr2,mum} and Novikov's
conjecture, much progress was made in the 1980s to
characterizing Jacobians and Pryms using Fay-like formulas and KP-like
equations.

Let us first describe the trisecant identity in geometric terms.
The Kummer variety $K(X)$ of $X\in\A_g$ is the image of the Kummer map
\beq\label{kum}
K=K_X\colon X\ni z\longmapsto
\bigl(:\Theta[\e,0](z):\bigr)\in \mathbb{CP}^{2^g-1}\,,
\eeq
where
$\Theta[\e,0](z)=\theta[\e,0](2z,2B)$ are the level two theta-functions
with half-integer characteristics $\e\in((1/2)\bZ/\bZ)^g$, i.e., they equal
$\theta(2(z+B\e),2B)$ up to some exponential factor so that we have
\beq\label{ThetaQuad}
\theta(z+w)\theta(z-w)=\sum_{\e\in((1/2)\bZ/\bZ)^g}\Theta[\e,0](z)\Theta[\e,0](w)\,.
\eeq
We have $K(-z)=K(z)$ and $K(X)\simeq X/\{\pm1\}$.

A \emph{trisecant\/} of the Kummer variety is a projective line which meets
$K(X)$ at three points. \emph{Fay's trisecant formula\/} states that if $X=J(\G)$, then $K(X)$ has a
family of trisecants parameterized by 4 points $A_i$, $1\le i\le4$, on $\G$.
Gunning proved in \cite{gun1} that, under certain non\-degeneracy
conditions, that the existence of \emph{one-parametric} family of trisecants characterizes the Jacobians.

Gunning's work was extended by Welters who proved that a Jacobian variety can
be characterized by the existence of a formal one-parameter family of flexes of the
Kummer variety \cite{wel}. A flex of the Kummer variety
is a projective line which is tangent to $K(X)$ at some point up to order 2.
It is a limiting case of trisecants when the three intersection points come together.

In \cite{arb-decon} Arbarello and De~Concini showed that the assumption in
Welters' characterization is equivalent to a infinite sequence of
partial differential equations known as the KP hierarchy, and proved
that only a first finite number of equations in the sequence are sufficient,
by giving an explicit bound for the number of
equations, $N=[(3/2)^gg!]$, based on the degree of $K(X)$.

An algebraic argument based on earlier results of Burchnall, Chaundy
and the author \cite{ch,kr1,kr2} characterizes the Jacobians
using a commutative ring $R$ of ordinary differential operators associated
to a solution of the KP hierarchy. A simple counting argument then shows
that only the first $2g+1$ time evolutions in the hierarchy are needed
to obtain $R$.
The $2g+1$ KP flows yield a finite number of differential equations for
the Riemann theta function $\theta$ of $X$, to characterize a Jacobian.
As for the number of equations, an easy estimate shows that $4g^2$ is
enough, although more careful argument should yield a better bound.

Novikov's conjecture, that just the first equation of the hierarchy ($N=1$!),
suffices to characterize the Jacobians, i.e.:

\medskip
\noindent
\emph{an indecomposable symmetric matrix $B$ with positive definite imaginary part is the period matrix of a basis of normalized holomorphic differentials on a smooth  algebraic curve $\G$ if and only if there are vectors $U\neq 0,V,W$, such that the function
\beq\label{u} u(x,y,t)=2\p_x^2 \ln \theta (Ux+Vy+Wt+Z|B),
\eeq
satisfies the KP equation \eqref{kp}}

\smallskip
\noindent
for quite some time seemed to be the strongest possible  characterization within the reach of the soliton theory.

\section{Welter's conjecture}

Novikov's conjecture is equivalent to the statement that the Jacobians are
characterized by the existence of length $3$ formal jet of flexes.

In \cite{wel1} Welters formulated the question: \emph{if the Kummer variety $K(X)$ has \emph{one\/}
trisecant, does it follow that $X$ is a Jacobian ?} In fact, there are
three particular cases of the Welters conjecture, corresponding to three possible configurations of
the intersection points $(a,b,c)$ of $K(X)$ and the trisecant:
\begin{itemize}
\item[(i)] all three points coincide $(a=b=c)$;
\item[(ii)] two of them coincide $(a=b\neq c)$;
\item[(iii)] all three intersection points are distinct
$(a\neq b\neq c\neq a)$.
\end{itemize}
Of course the first two cases can be regarded as degenerations of the general case~(iii).
However, when the existence of only one trisecant is assumed, all three cases are
independent and require its own approach. The approaches used in \cite{kr-schot,kr-tri} were based on the theories of three main soliton hierarchies (see details in \cite {kr-shiot}): the KP hierarchy for (i), the 2D Toda hierarchy for (ii) and the Bilinear Discrete Hirota Equations (BDHE) for (iii).
Recently, pure algebraic proof of the first two cases of the trisecant conjecture were obtained in \cite{acp}.

\begin{theorem}\label{th1.1} An indecomposable principally polarized abelian variety $(X,\theta)$
is the Jacobian variety of a smooth algebraic curve of genus g if and only if there exist $g$-dimensional vectors
$U\neq 0, V,A $ , and constants $p$ and $E$ such that one of the following three equivalent conditions are satisfied:

$(A)$ equality \eqref{intCM} with $\tau=\theta (Ux+Vt+Z)$ and
\beq\label{intpsi}\psi=\frac{\theta(A+Ux+Vt+Z)} {\theta(Ux+Vt+Z)}\, e^{p\,x+E\,t}
\eeq
holds, for an arbitrary vector $Z$;

\medskip
$(B)$ for all theta characteristics $\e\in (\frac 1 2\bZ/\bZ)^g$
$$ 
\left(\p_V-\p_U^2-2p\,\p_U+(E-p^2)\right)\, \Theta[\e,0](A/2)=0
$$ 
(here and below $\p_U$, $\p_V$ are the derivatives along the vectors $U$ and $V$, respectively).

\medskip
$(C)$ on the theta-divisor $\Theta=\{Z\in X\,\mid\, \theta(Z)=0\}$ the equation
\beq\label{cm70}
[(\p_V\theta)^2-(\p_U^2\theta)^2]\p_U^2\theta
+2[\p_U^2\theta\p_U^3\theta-\p_V\theta\p_U\p_V\theta]\p_U\theta+
[\p_V^2\theta-\p_U^4\theta](\p_U\theta)^2=0
\eeq
holds.
\end{theorem}
The direct substitution of the expression (\ref{intpsi}) in equation (\ref{intCM}) and the use of the addition formula for the Riemann theta-functions shows the equivalence of conditions $(A)$ and $(B)$ in the theorem.
Condition $(B)$ 
means that the image of the point $A/2$ under the Kummer map is an inflection point (case~(i) of Welters' conjecture).

Condition $(C)$, which we call the abelian CM system is the relation that is \emph{really used\/} in the proof of the theorem. Formally it is weaker than the other two conditions because its derivation does not use an explicit form of the solution $\psi$ of equation (\ref{intCM}), but requires only that $\psi$ is \emph{meromorphic solution}. The latter, as we have seen, implies equation \eqref{cm50}. Expanding the function $\theta$
in a neighborhood of a point $z\in\Theta:=\{z\mid\theta(z)=0\}$ such that
$\p_U\theta(z)\ne0$, and noting that the latter condition holds on a dense subset of $\Theta$ since $B$ is indecomposable, it is easy to see that equation (\ref{cm50}) is equivalent to (\ref{cm70}).

Equation (\ref{intCM}) is one of the two auxiliary linear problems for the KP equation.
For the author, the motivation to consider not the whole KP equation but just one of its auxiliary linear problem was his earlier work \cite{kr3} on the elliptic Calogero-Moser (CM) system, where it was observed for the first time that equation (\ref{intCM}) is all what one needs to construct the elliptic solutions of the KP equation.

\medskip
The proof of Welters' conjecture was completed in \cite{kr-tri}. First, here is
the theorem which treats case~(ii) of the conjecture:

\begin{theorem}\label{th1.2} An indecomposable, principally polarized abelian variety $(X,\theta)$
is the Jacobian of a smooth curve of genus g if and only if
there exist non-zero $g$-dimensional vectors
$U\neq A \pmod \Lambda$, $V$ constants $p,E$, such that one of the following equivalent conditions
holds:

$(A)$  equation \eqref{intRS} with $\tau=\theta(Ux+Vt+Z)$ and $\psi$ as in \eqref{intpsi}
holds for an arbitrary $Z$.

\medskip
$(B)$ The equations
$$ 
\p_{V}\Theta[\e,0]\left((A-U)/2\right)-e^{p}\Theta[\e,0]\left((A+U)/2\right)
+E\Theta[\e,0]\left((A-U)/2\right)=0,
$$
are satisfied for all $\e\in({\frac 1 2}\bZ/\bZ)^g$. Here and below $\p_V$ is the constant
vector field on $\mathbb{C}^g$ corresponding to the vector $V$.

\medskip
$(C)$ The equation
\beq\label{cm7}
\p_V\left[\theta(Z+U)\,\theta(Z-U)\right]\p_V\theta(Z)=
\left[\theta(Z+U)\,\theta(Z-U)\right]\p^2_{VV}\theta(Z)
\eeq
is valid on the theta-divisor $\Theta=\{Z\in X\,\mid\, \theta(Z)=0\}$.
\end{theorem}
Recall, that equation (\ref{intRS}) is one of the two auxiliary linear problems for the $2D$ Toda lattice
equation \eqref{2DT}. The idea to use
it for the characterization of the Jacobians was motivated by \cite{kr-schot} and the author's earlier work with Zabrodin \cite{zab}, where
a connection of the theory of elliptic solutions of the $2D$ Toda lattice equations
and the theory of the elliptic Ruijsenaars-Schneider system was established.

The statement $(B)$ is the second particular case of
the trisecant conjecture: the line in $\mathbb{CP}^{2^g-1}$ passing through the points
$K((A-U)/2)$ and $K((A+U)/2)$ of the Kummer variety is tangent
to $K(X)$ at the point $K((A-U)/2)$.

The condition $(C)$ is what we call the abelian RS equation.

\medskip
The affirmative answer to the third particular case, (iii), of Welters' conjecture is given by the following statement.

\begin{theorem}\label{th1.3} An indecomposable, principally polarized abelian variety $(X,\theta)$
is the Jacobian of a smooth curve of genus g if and only if
there exist non-zero $g$-dimensional vectors
$U\neq V\neq A \neq U\, (\bmod \Lambda)$ such that one of the following equivalent
conditions holds:

$(A)$ equation \eqref{intBA} with $\tau_n(x)=\theta(xU+nV+Z)$ and
\beq\label{pd}
\psi_n(x)=\frac{\theta(A+xU+nV+Z)}{\theta(xU+nV+Z)}\, e^{xp+nE},
\eeq
holds for an arbitrary $Z$.

\medskip
$(B)$ The equations
$$ 
\Theta[\e,0]\Big(\frac{A-U-V} 2\Big)+e^{p}\Theta[\e,0]\Big(\frac{A+U-V}2\Big)
=e^E\Theta[\e,0]\Big(\frac{A+V-U}2\Big),
$$ 
are satisfied for all $\e\in({\frac 1 2}\bZ/\bZ)^g$.

\medskip
$(C)$ The equation
\beq\label{cm7d}
\frac{\theta(Z+U)\,\theta(Z-V)\,\theta(Z-U+V)}{\theta(Z-U)\,\theta(Z+V)\,\theta(Z+U-V)}=-1 \pmod \theta
\eeq
is valid on the theta-divisor $\Theta=\{Z\in X\,\mid\, \theta(Z)=0\}$.
\end{theorem}
Under the  assumption that the vector $U$ \emph{spans an elliptic curve} in $X$,
Theorem~\ref{th1.3} was proved in \cite{krwz}, where the connection of the
elliptic solutions of BDHE and, the so-called, elliptic nested Bethe nsatz equations
was established. The condition $(C)$ is its abelian generalization.

\section{The problem of characterization of Prym varieties}

An involution $\sigma: \G\to \G$ on a smooth algebraic curve $\G $ naturally determines
an involution $\s^*: J(\G)\longmapsto J(\G)$ on its Jacobian. The odd subspace with respect
to this involution is a sum of an Abelian subvariety of lower dimension, called the
Prym variety, and a finite group. The restriction of the principal polarization of
the Jacobian determines a polarization of the Prym variety which is principal if
and only if the original involution of the curve has at most two fixed points.
The problem of characterizing the locus $\P_g$ of Prym varieties of dimension $g$
in the space $\A_g$ of all principally polarized Abelian varieties is well known and
during its history has attracted considerable interest. This
problem is much harder than the Riemann-Schottky problem and until relatively
recently its solution in terms of a finite system of equations was completely open.

The problem of characterizing Prym varieties in the case of curves with an involution
having two fixed points was solved in \cite{kr-prym} in terms of the Schr\"odinger operators
integrable with respect to one energy level. The theory of such operators was
developed by Novikov and Veselov in \cite{nv1,nv2}, where the authors also introduced the
corresponding non-linear equation, the so-called Novikov-Veselov equation.
Curves with an involution having a pair of fixed points can be regarded as a limit
of unramified covers. A characterization of the Prym varieties in the latter case in
terms of the existence of quadrisecants was obtained  the
author and Grushevsky in (\cite{kr-quad}).

The existence of families of quadrisecants for curves with an involution having at
most two fixed points was proved in \cite{bd,fay2}. An analogue of Gunning's theorem
asserting that the existence of a family of secants characterizes Prym varieties
was proved by Debarre \cite{deb}. We note that the existence of one quadrisecant does
not characterize Prym varieties. A counterexample to the naive generalization of
Welters' conjecture was constructed by Beauville and Debarre in the work \cite{bd}.

It was proved in (\cite{kr-quad}) that the existence of a symmetric pair of quadrisecants is
a characteristic property for Prym varieties of unramified covers.
\begin{theorem} [Geometric characterization of Prym varieties.] An indecomposable
principally polarized Abelian variety $(X,\theta)\in\A_g$ is in the closure of the locus of
Prym varieties of smooth unramified double covers if and only if there exist four
distinct points $p_1,p_2,p_3,p_4\in X$, none of them of order two, such that the images of
the Kummer map of the eight points $p_1\pm p_2\pm p_3\pm p_4$ lie on two quadrisecants (the
corresponding quadruples of points are determined by the number of plus signs ).
\end{theorem}
We should note that the proof of this statement required constructing and developing
the theory of a new integrable equation because before that, in contrast with
all other cases, no non-linear equations whose algebro-geometric solutions are associated
to unramified double covers were known.

The auxiliary linear equation of the corresponding analogue of the Novikov-
Veselov equation is a discrete analogue of the potential Schr\"o\-dinger equation considered first in \cite{grin}. It
has the form
\begin{equation}\label{laxdd0}
\psi_{n+1,m+1}-u_{n,m}(\psi_{n+1,m}-\psi_{n,m+1})-\psi_{n,m}=0
\end{equation}
The analog of the condition $(C)$ in the previous theorem which is also can be thought as the abelian generalization of some discrete time integrable system (which has not been studied so far) is as follows:

$(C)$ \emph{There are constants $c_i^\pm, i=1,2,3$ such that two equations (one for the top choice of signs everywhere,
and one --- for the bottom)
\begin{eqnarray}\label{cm7dd}
c_1^{\mp 2}c_3^2\ \theta(Z+U-V)\,\theta(Z-U\pm W)\,\,\theta(Z+V\pm W)&\nonumber\\
+c_2^{\mp 2}c_3^2\ \theta(Z-U+V)\,\theta(Z+U\pm W)\,\,\theta(Z-V\pm W)\nonumber\\
=c_1^{\mp 2}c_2^{\mp 2}\,\theta(Z-U-V)\,\theta(Z+U\pm W)\,\,\theta(Z+V\pm W)&\nonumber\\
+\theta(Z+U+V)\,\theta(Z-U\pm W)\,\,\theta(Z-V\pm W)& \end{eqnarray}
 are valid on the theta divisor} $\{Z\in X: \theta(Z)=0\}$.
s
\section{Abelian solutions of the soliton equations}

The general concept of \emph{abelian solutions} of soliton equations was introduced by T.~Shiota and the author in \cite{kr-shio,kr-shio1}. It provides a unifying framework for the theory of the elliptic solutions of these equations and algebraic-geometrical solutions of rank 1 expressible in terms of Riemann (or Prym) theta-function.
A solution $u(x,y,t)$
of the KP equation is called \emph{abelian\/} if it is of the form
\beq\label{ushio}
u=-2\p_x^2\ln \tau(Ux+z,y,t)\,,
\eeq
where $x$, $y$, $t\in\mathbb C$ and $z\in \mathbb C^n$ are independent variables,
$0\ne U\in\mathbb C^n$, and for all $y$, $t$ the function
$\tau(\cdot,y,t)$ is a holomorphic section of a line bundle $\L=\L(y,t)$ on an
abelian variety $X=\bC^n/\Lambda$, i.e., for all $\l\in\Lambda$
it satisfies the monodromy relations \eqref{tau1}.

In the case of sections of the canonical line bundle on a principally polarized
Abelian variety the corresponding theta-function is unique up to normalization.
Hence the ansatz \eqref{ushio} takes the form $u=-2\p_x^2\ln\theta(Ux+Z(y,t)+z)$. Since
flows commute with each other, the dependence of the vector $Z(y, t)$ must be linear:
\beq\label{uan1}
u=-2\p_x^2 \ln\theta(Ux+Vy+Wt+z)\,.
\eeq
Therefore, the problem of classification of such Abelian solutions is the same problem
as posed by Novikov.

In the case of one-dimensional Abelian varieties the problem of classification of
Abelian solutions is the problem of classification of the elliptic solutions.
The theory of elliptic solutions of the KP equation goes back to the remarkable
work \cite{amkm}, where it was found that the dynamics of poles of the elliptic (rational
or trigonometric) solutions of the Korteweg-de Vries equation can be described
in terms of the elliptic (rational or trigonometric) Calogero-Moser (CM) system
with certain constraints. It was observed in \cite{kr3} that, when the constraints are
removed, this restricted correspondence becomes an isomorphism when the elliptic
solutions of the KP equation are considered.
The elliptic solutions of the KP equation are  distinguished amongst the general algebraic-geometric solutions
by the condition that the corresponding vector $U$ spans an elliptic curve embedded into the Jacobian of the curve.
Note that, for any vector $U$, the closure of the group $\{Ux|\,x\in \mathbb C,\}$ is an Abelian
subvariety $X\subset J(\G)$. So when this closure does not coincide with the whole
Jacobian, we get non-trivial examples of Abelian solutions. Briefly, the main result
on the classification of Abelian solutions of KP obtained in \cite{kr-shio} can be formulated
as the statement that all the Abelian solutions are obtained in this manner.
To avoid some technical complications we give the formulation of the corresponding
theorem in the situation of general position.

\begin{theorem} Let $u(x,y,t)$ be an abelian solution of the KP such that the group $\bC U\bmod\Lambda$ is dense in $X$. Then there exists a unique algebraic curve $\G$ with smooth marked point
$P\in\G$, holomorphic imbedding
$j_0\colon X\to J(\G)$ and a torsion-free rank 1 sheaf $\mathcal F\in\overline{{Pic}^{g-1}}(\G)$
where $g=g(\G)$ is the arithmetic genus of $\G$, such that setting with the notation
$j(z)=j_0(z)\otimes\mathcal F$
\beq\label{is1}
\tau(Ux+z,y,t)=\rho(z,y,t)\,\widehat\tau(x,y,t,0,\ldots\mid\G,P,j(z))
\eeq
where $\widehat\tau(t_1,t_2,t_3,\ldots \mid \G,P,\mathcal F)$ is the KP $\tau$-function
corresponding to the data $(\G,P,\mathcal F)$, and  $\rho(z,y,t)\not\equiv0$
 satisfies the condition $\p_U\rho=0$.
\end{theorem}
 Note that if $\G$ is smooth then:
\beq\label{is2}
\widehat\tau(x,t_2,t_3,\dots\mid\G,P,j(z))=
\theta\Bigl(Ux+\sum V_it_i+j(z)\Bigm|B(\G)\Bigr)\,
e^{Q(x,t_2,t_3,\ldots)}\,,
\eeq
where $V_i\in\mathbb C^n$, $Q$ is a quadratic form, and $B(\G)$ is the
period matrix of $\G$.
A linearization on $J(\G)$ of the nonlinear $(y,t)$-dynamics for  $\tau(z,y,t)$ indicates the possibility of the existence of integrable systems on spaces of theta-functions of higher level. A CM system is an example of such a system for $n=1$.

\section{The Baker-Akhiezer functions -- General scheme}

The "only if" part of all the theorems above is a corollary of the general algebraic-geometric construction
of solutions of soliton equations based on a concept of the Baker-Akhiezer function.

Let $\Gamma$ be a nonsingular algebraic curve of genus $g$ with $N$ marked points
$P_{\a}$ and fixed local parameters $k_{\a}^{-1}(p)$ in  neighborhoods
of the marked points. The basic scalar \emph{multi-point\/} and \emph{multi-variable} Baker-Akhiezer function $\psi(t,p)$ is a function of external parameters
\beq\label{times}
t=(t_{\a,i}),\ \a = 1,\ldots, N ; \ i=0,\ldots ;\ \  \sum_\a t_{\a,0}=0,
\eeq
only finite number of which is non-zero, and a point $p\in \G$. For each set of the external parameters $t$ it is defined by its analytic properties on $\G$.

\emph{Remark.} For the simplicity we will begin with the assumption that the variables $t_{\a,0}$ are integers, i.e.,
$t_{\a,0}\in \mathbb Z$.

\begin{lemma} For any set of $g$ points  $\gamma_1,\ldots,\gamma_g$ in a
general position there exists a unique (up to constant factor
$c(t)$) function $\psi (t,p)$,
such that:

(i) the function $\psi$ (as a function of the variable $p\in \G$) is meromorphic everywhere except for the points $P_{\a}$ and
has at most simple poles at the points $\gamma_1,\ldots,\gamma_g$ ( if all
of them are distinct);

(ii) in a neighborhood of the point $P_{\a}$ the function $\psi$ has the
form
\beq
\psi (t,p) =k_\a^{t_{\a,0}} \exp \biggl(\sum_{i=1}^{\infty} t_{\a ,i} k_{\a}^{i}
\biggr) \biggl( \sum_{s=0}^{\infty} \xi_{\a,s}(t) k_{\a}^{-s} \biggr), \quad k_\a=k_\a(p)
\label{2.1}
\eeq
\end{lemma}
 From the uniqueness of the Baker-Akhiezer function it follows that:

\begin{theorem}\label{th2.1}
For each pair
$(\a,\,n>0)$ there exists a unique operator $L_{\a ,n}$ of the form
\beq
L _{\a ,n} = \p _{\a,1}^{n}
+ \sum_{j=0}^{n-1} u_{j}^{(\a ,n)}(t) \p_{\a ,1}^{j},
\label{2.2}
\eeq
(where $ \p _{\a,n} =\p / \p t _{\a ,n}$)
such that
\beq
\left(\p_{\a,n} - L_{\a,n}\right)\, \psi (t,p) = 0 .
\label{2.3}
\eeq
\end{theorem}
The idea of the proof of the theorems of this type proposed
in \cite{kr1}, \cite{kr2} is universal.

For any formal series of the form (\ref{2.1}) their exists a unique operator
$L_{\a ,n} $ of the form (\ref{2.2}) such that
\beq
\left(\p_{\a ,n} - L_{\a ,n} \right)\, \psi (t,p) = O(k_\a^{-1})
\exp \,\biggl(\sum_{i=1}^{\infty} t_{\a ,i} k_{\a}^{i} \biggr) . \label{2.4}
\eeq
The coefficients of $L_{\a ,n} $ are universal differential polynomials with
respect to $\xi_{s,\a }$. They can be found after substitution of the
series (\ref{2.1}) into (\ref{2.4}).

It turns out that if the series (\ref{2.1}) is not formal but is an
expansion of the Baker-Akhiezer function in the neighborhood of $P_{\a}$ the
\emph{congruence (\ref{2.4}) becomes an equality}. Indeed, let us consider the
function $\psi_{1}$
\beq
\psi_{1} = (\p_{\a ,n} - L_{\a ,n}) \psi (t,p).
\label{2.5}
\eeq
It has the same analytic properties as $\psi$ except for the only one.
The expansion of this function in the neighborhood of $P_{\a}$ starts
from $O(k_\a^{-1})$. From the uniqueness of the Baker-Akhiezer function it follows that
$\psi_1 = 0 $ and the equality (\ref{2.3}) is proved.

\begin{corollary} The operators $ L_{\a ,n}$ satisfy the compatibility
conditions
\beq
\bigl[ \p_{\a ,n} - L_{\a ,n} ,
\p_{\a ,m} - L_{\a ,m} \bigr] = 0 .\label{2.6}
\eeq
\end{corollary}
\noindent
The equations (\ref{2.6}) are gauge invariant. For any function
$c(t)$ operators
\beq
\wt L_{\a ,n} = c L_{\a ,n} c^{-1} +
( \p_{\a ,n}c) c^{-1} \label{2.7}
\eeq
have the same form (\ref{2.2}) and satisfy the same operator equations
(\ref{2.6}). The gauge transformation (\ref{2.7}) corresponds to the gauge
transformation of the Baker-Akhiezer function
\beq
\wt\psi (t,p) = c(t) \psi (t,p) \label{2.7a}
\eeq
In addition to differential equations (\ref{2.3}) the Baker-Akhiezer function satisfies an infinite system of
differential-difference equations. Recall that the discrete variables $t_{\a,0}$ are subject to the constraint
$\sum_\a t_{\a,0}=0$. Therefore, only the first $(N-1)$ of them are independent and $t_{N,0}=-\sum_{\a=1}^{N-1}t_{\a,0}$. Let us denote by $T_{\a},\ \ \a=1,\ldots,N-1,$ the operator that shifts
the arguments $t_{\a,0}\to t_{\a,0}+1$ and $t_{N,0}\to t_{N,0}-1$, respectively. For the sake of brevity in the formulation of the next theorem we introduce the operator $T_N=T_1^{-1}$.
\begin{theorem}
For each pair
$(\a,\,n>0)$ there exists a unique operator $\wh L_{\a ,n}$ of the form
\beq
\wh L_{\a ,n} = T_{\a}^{n}+ \sum_{j=0}^{n-1} v_{j}^{(\a ,n)}(t) \,T_{\a}^{j},\ \
v_{0}^{(N ,n)}(t)=0.
\label{2.2d}
\eeq
such that
\beq
\left(\p_{\a,n} - \wh L_{\a,n}\right)\, \psi (t,p) = 0 .
\label{2.3d}
\eeq
\end{theorem}
The proof is identical to that in the differential case.
\begin{corollary} The operators $\wh L_{\a ,n}$ satisfy the compatibility
conditions
\beq
\bigl[ \p_{\a ,n} - \wh L_{\a ,n} ,
\p_{\a ,m} - \wh L_{\a ,m} \bigr] = 0 .\label{2.6d}
\eeq
\end{corollary}

\subsection*{Theta-functional formulae}
It should be emphasized that the algebro-geometric construction
is not a sort of abstract ``existence'' and ``uniqueness'' theorems. It
provides the explicit formulae for solutions in terms of the Riemann
theta-functions. They are the corollary of the explicit formula for the
Baker-Akhiezer function.

Let $a_i,b_i\in H_1(\G,\bZ)\, i=1\ldots,g, $ be a basis of cycles on $\G$ with the canonical intersection matrix, i.e. $a_i\cdot a_j=b_i\cdot b_j=0, a_i\cdot b_j= \delta_{ij}$ and let $\omega_i$ be the basis of holomorphic differentials on $\G$ normalized by the equations $\oint_{a_j}\omega_j=\delta_{ij}$. The matrix $B$ of their $b$-periods $B_{ij}=\oint_{b_i} \omega_j$ is indecomposable symmetric matrix with positive definite imaginary part. By formula \eqref{teta1} it defines the Riemann theta-function $\theta(z)=\theta(z|B)$.
.

\begin{theorem} The Baker-Akhiezer function is given by the formula
\beq
\psi(t,p)=c(t)\exp\left(\sum t_{\a,i}\Omega_{\a,i}(p)\right)
\frac{\theta(A(p)+\sum U_{\a,i}t_{\a,i}+Z)}{\theta(A(p)+Z)} \label{2.101}
\eeq
Here the sum is taken over all the indices $(\a, i>0)$ and over the indices $(\a,0)$ with $\a=1,\ldots,N-1$, and:

a) $\Omega_{\a,i}(p)$ is the abelian integral,
$
\Omega_{\a,i}(p)=\int^p d\Omega_{\a,i},
$
corresponding to the unique normalized,
$
\oint_{a_k} d\Omega_{\a,i}=0,
$
meromorphic differential on
$\Gamma$, which for $i>0$ has the only pole of the form
$
d\Omega_{\a,i}=d\left(k_{\a}^i+O(1)\right)
$
at the marked point $P_{\a}$ and for $i=0$ has simple poles at the marked point $P_{\a}$ and $P_{N}$ with residues
$\pm 1$, respectively;

b) $2\pi iU_{\a,j}$ is the vector of $b$-periods of the differential
$d\Omega_{\a, j}$, i.e.,
$$
U_{\a,j}^k=\frac 1{2\pi i} \oint_{b_k} d\Omega_{\a,j};
$$

c) $A(p)$ is the Abel transform, i.e., a vector with the coordinates
$
A_i(p)=\int^p d\omega_i
$

d) $Z$ is an arbitrary vector (it corresponds to the divisor of poles of
Baker-Akhiezer function).
\end{theorem}
Notice, that from the bilinear Riemann relations it follows that the expansion of the Abel transform near the marked point has the form
\beq\label{abelexpans}
A(p)=A(P_{\a})-\sum_{i=1}^\infty \frac{1}{i}U_{\a,i}k_{\a}^{-i}
\eeq

\subsubsection*{Example 1. One-point Baker-Akhiezer function. KP hierarchy}

In the one-point case the Baker-Akhiezer function has an exponential
singularity at a single point $P_1$ and depends on a single set of variables $t_i=t_{1,i}$. Note that in this case there is no discrete variable, $t_{1,0}\equiv 0$.
Let us  choose the normalization of the Baker-Akhiezer function with the
help of the condition $\xi_{1,0}= 1 $, i.e., an expansion of $\psi$ in the neighborhood
of $P_1$ equals
\beq
\psi (t_1,t_2,\ldots,p) =
\exp \biggl(\sum_{i=1}^{\infty} t_{i} k^{i} \biggr)
\biggl(1+ \sum_{s=1}^{\infty} \xi_{s}(t) k^{-s}\biggr).\label{2.8}
\eeq
Under this normalization (gauge) the corresponding operator $L_n$ has the form
\beq
L_n =\p_1^n +\sum_{i=0}^{n-2} u_i^{(n)} \p_1^i . \label{2.9}
\eeq
For example, for $n=2,3$ after redefinition $x=t_1$ we have
\beq \label{kp100}
L_2=\p_x^2-u, \qquad L_3=\p_x^3-\frac32u\p_x-w
\eeq
with $\quad u=2\p_x \xi_1, w=3\p_x\xi_2+3\p^2_x\xi_1-\frac 32 u\xi_1$.

If we define  $y=t_2, t=t_3$, then from \eqref{2.6} with $n=2,m=3$ it follows
$ u(x,y,t,t_4,\ldots )$ satisfies the KP equation \eqref{kp}.

The normalization of the leading coefficient in (\ref{2.8}) defines the the function $c(t)$ in (\ref{2.101}).
That gives the following formula for the normalized one-point Baker-Akhiezer function:
\beq
\psi(t,p)=\exp\left(\sum t_{i}\Omega_{i}(p)\right)
\frac{\theta(A(p)+\sum U_{i}t_{i}+Z) \,\theta(Z)}{\theta(\sum U_{i}t_{i}+Z)\,\theta(A(p)+Z)} , \label{BA1}
\eeq
(shifting $Z$ if needed we may assumed that $A(P_1)=0$).
In order to get the explicit theta-functional form of the solution of the KP
equation it is enough to take the derivative of the first coefficient of the
expansion at the marked point of the ratio of theta-functions in the formula (\ref{BA1}).

Using (\ref{abelexpans}) we get the final formula for the algebro-geometric solutions of the KP hierarchy \cite{kr2}
\beq
u(t_1,t_2, \ldots ) = - 2 \p_1^2 \ln \theta \left(\sum_{i=1}^{\infty} U_i t_i +Z\right) +
\hbox{const}. \label{2.11}
\eeq

\subsubsection*{Example 2. Two-point Baker-Akhiezer function. $2D$ Toda hierarchy}

In the two-point case the Baker-Akhiezer function has exponential
singularities at two points $P_\a, \a=1,2,$ and depends on two sets of continuous variables $t_{\a,i>0}$. In
addition it depends on one discrete variable $n=t_{1,0}=-t_{2,0}$.
Let us  choose the normalization of the Baker-Akhiezer function with the
help of the condition $\xi_{1,0}= 1$.

According to Theorem~\ref{th2.1}, the function $\psi$ satisfies two sets of differential equations. The compatibility
conditions (\ref{2.6}) within the each set can be regarded as two copies of the KP hierarchies. In addition
the two-point Baker-Akhiezer function satisfies differential-difference equations (\ref{2.2d}). The first two of them have the form
\beq\label{nov23}
(\p_{1,1}-T+u)\psi=0,\ \ \ (\p_{2,1}-w T^{-1})\psi=0,
\eeq
where
\beq\label{nov231}
u=(T-1)\xi_{1,1}(n,t) , \ \ \ w=e^{\varphi_n-\varphi_{n-1}},\ \ e^{\varphi_n(t)}=\xi_{2,0}(n,t)
\eeq
The compatibility condition of these equations is equivalent to the $2D$ Toda equation
equation with $\xi=t_{1,1}$ and $\eta=t_{2,1}$. The explicit formula for the solution $\varphi_n(t)$ is a direct corollary of the explicit formula for the Baker-Akhiezer function:
\beq\label{2dsolution}
\varphi_n(t_{\a,i>0})=\ln\frac{\theta((n+1)U+\sum U_{\a,i}t_{\a,i}+Z)}{
\theta(nU+\sum U_{\a,i}t_{\a,i}+Z)}\,, \a=1,2
\eeq

\subsubsection*{Example 3. Three-point Baker-Akhiezer function}

Starting with three-point case, in which the number of discrete variables is $2$, the Baker-Akhiezer function satisfies certain linear difference equations (in addition to the differential and the differential-difference equations (\ref{2.3}), (\ref{2.3d})). The origin of these equations is easy to explain. Indeed, if all the continuous variables vanish, $t_{\a,i>0}=0$, then the Baker-Akhiezer function $\psi_{n,m}:=\psi(n,m,p)$, where $n=-t_{1,0}$, $m=-t_{2,0}$, is a meromorphic function having pole of order $n+m$ at $P_3$ and
zeros of order $n$ and $m$ at $P_1$ and $P_2$ respectively, i.e.,
\beq\label{H0}
\psi_{n,m}\in H^0(D+n(P_3-P_1)+m(P_3-P_2)),\ \   D=\g_1+\cdots+\g_g
\eeq
The functions $\psi_{n+1,m},\psi_{n,m+1},\psi_{n,m}$ are all in the linear space $H^0(D+(n+m+1)P_3-nP_1-mP_2)$.
By Riemann-Roch theorem for a generic $D$ the latter space is $2$-dimensional. Hence, these functions are linear dependent, and they can be normalized such their linear dependence takes the form
\beq\label{laxdd}
\psi_{m,n+1}=\psi_{m+1,n}+u_{m,n}\psi_{m,n}
\eeq
with
\beq\label{uformula}
u_{n,m}= \frac{\tau_{m+1,n+1}\tau_{m,n}}{\tau_{m,n+1}\tau_{m+1,n}}, \quad \tau_{m,n}:=\theta(mU+nV+Z)
\eeq
For the first glance it seems that everything here is within the framework of classical algebraic-geometry. What might be new brought to this subject by the soliton theory is understanding
that \emph{the discrete variables $t_{\a,0}$ can be replaced by continuous ones}. Of course, if
in the formula (\ref{2.101}) the variable $t_{\a,0}$ is not an integer, then $\psi$ is not a single valued function on $\G$. Nevertheless, because the monodromy properties of $\psi$ do not change if the shift of the argument is integer, it satisfied the same type of linear equations with coefficients given by the same type of formulae. It is necessary to emphasize that in such a form the difference equation becomes \emph{functional} equation.

\medskip\noindent
In the four-point case there is three discrete variables $n$, $m$, $l$. In each two of them the Baker-Akhiezer function satisfies
a difference equation. Compatibility of these equations is the BDHE equation

\section{Key idea and steps of the proofs.}\label{s:proofs}

As it was mentioned above the proof of all the particular cases of Welters' trisecant conjecture
uses different hierarchies: the KP, the 2D Toda, and BDHE. In each case there are some specific difficulties but
the main ideas and structures of the proof are the same. As an instructive example we present in this section the idea and key steps of the proof of the first
particular case of Welters' conjecture, namely, the proof of Theorem~\ref{th1.1}.

As it was mentioned above the implication $(A)\to (C)$ is a direct corollary of \eqref{cm50}.
Now we are going to show that \eqref{cm50}, which is satisfied when \eqref{intCM} has \emph{one} meromorphic solution, is sufficient  for the  existence of \emph{one-parametric} family of formal wave solutions below.

 \medskip
The wave solution of (\ref{intCM}) is a solution of the form
\beq\label{ps}
\psi(x,y,k)=e^{kx+(k^2+b)t}\biggl(1+\sum_{s=1}^{\infty}\xi_s(x,t)\,k^{-s}\biggr)\,.
\eeq
\begin{lemma}
Suppose that equations (\ref{cm50}) for the zeros of $\tau(x,t)$ hold. Then there exist
meromorphic wave solutions of equation (\ref{intCM}) that have simple poles at zeros $q$ of $\tau$ and
are holomorphic everywhere else.
\end{lemma}
\begin{proof}Substitution of (\ref{ps}) into (\ref{intCM}) gives a recurrent system of
equations
\beq\label{xis}
2\xi_{s+1}'=\p_t\xi_s+u\xi_s-\xi_s''
\eeq
We are going to prove by induction that this system has meromorphic solutions with
simple poles at all the zeros $q$ of $\tau$.

Let us expand $\xi_s$ at $q$:
\beq\label{5}
\xi_s=\frac{r_s}{x-q}+r_{s0}+r_{s1}(x-q)+\ldots\,,
\eeq
Suppose that $\xi_s$ is defined and equation (\ref{xis}) has a meromorphic solution.
Then the right-hand side of (\ref{xis}) has the zero residue at $x=q$, i.e.,
\beq\label{res}
res_{q}\left(\dot \xi_s+u\xi_s-\xi_s''\right)=\dot r_s+v_ir_s+2r_{s1}=0
\eeq
We need to show that the residue of the next equation vanishes also.
 From (\ref{xis}) it follows that the coefficients of the Laurent expansion for $\xi_{s+1}$
are equal to
\beq\label{6}
r_{s+1}=-\dot q r_s-2r_{s0},\quad
 2r_{s+1,1}=\dot r_{s0}-r_{s1}+w r_s+v r_{s0}\,.
\eeq
These equations imply
\beq
\dot r_{s+1}+v r_{s+1}+2r_{s+1,1}=-r_s(\ddot q-2w)-\dot q(\dot r_s+v r_s+2r_{s1})=0,
\eeq
and the lemma is proved.
\end{proof}

\subsection*{$\lambda$-periodic wave solutions}

Our next step in the proof is to fix a \emph{translation-invariant} normalization of $\xi_s$
which defines wave functions uniquely up to a $x$-independent factor.
It is instructive to consider first the case of the periodic potentials $u(x+1,t)=u(x,t)$
(see details in \cite{kp}).

Equations (\ref{xis}) are solved recursively by the formulae
\beq
\xi_{s+1}(x,t)=c_{s+1}(t)+\xi_{s+1}^0(x,t)\,,\label{kp1}
\eeq
\beq\label{kp2}
\xi_{s+1}^0(x,t)=\frac 12\int_{x_0}^x (\dot \xi_s-\xi_s''+u\xi_s)\,dx=0\, ,
\eeq
where $c_s(t)$ are \emph{arbitrary} functions of the variable $t$.
Let us show that the periodicity condition $\xi_s(x+1,t)=\xi_s(x,t)$
defines the functions $c_s(t)$ uniquely up to an additive constant.
Assume that $\xi_{s-1}$ is known  and satisfies the condition that the corresponding
function $\xi_s^0$ is periodic.
The choice of the function $c_s(t)$ does not affect the periodicity property of
$\xi_s$, but it does affect the periodicity in $x$ of the function
$\xi_{s+1}^0(x,t)$. In order to make  $\xi_{s+1}^0(x,t)$ periodic,
the function $c_s(t)$ should satisfy the linear differential equation
\beq\label{kp4}
\p_t c_s(t)+B(t)\,c_s(t)+\int_{x_0}^{x_0+1} \left(\dot\xi_s^0(x,t)+
u(x,t)\,\xi_s^0(x,y)\right)\,dx\ ,
\eeq
where $B(t)=\int_{x_0}^{x_0+1} u\, dx$.
This defines $c_s$ uniquely up to a constant.

In the general case, when $u$ is quasi-periodic, the normalization of the wave functions
is defined along the same lines.

Let $Y_U=\langle \bC U\rangle$ be the Zariski closure of the group
$\bC U=\{Ux\mid x\in\bC\}$ in $X$. Shifting $Y_U$ if needed, we may assume, without loss of generality, that
$Y_U$ is not in the singular locus $\Sigma$ defined as $\p_U$-invariant subset of the theta-divisor
$\Theta$, i.e. $Y_U\not\subset\Sigma$. Then, for a sufficiently small
$t$, we have $Y_U+Vt\notin\Sigma$ as well.
Consider the restriction of the theta-function onto the affine subspace
$\mathbb C^d+Vt$, where
$\mathbb C^d:=($the identity component of $\pi^{-1}(Y_U))$, and
$\pi\colon \mathbb C^g\to X=\mathbb C^g/\Lambda$ is the universal covering map of $X$:
\beq\label{ttt1}
\tau (z,t)=\theta(z+Vt), \ \ z\in \mathbb C^d.
\eeq
The function $u(z,t)=-2\partial_U^2\ln \tau$ is periodic with respect to the lattice
$\Lambda_U=\Lambda\cap \mathbb C^d$ and, for fixed $t$, has a double pole along the divisor
$\Theta^{\,U}(t)=\left(\Theta-Vt\right)\cap \mathbb C^d$.

\begin{lemma}\label{lem6.3} Let equations (\ref{cm50}) for zeros of $\tau(Ux+z,t)$ hold.
Then:

(i) equation (\ref{intCM}) with the potential $u(Ux+z,t)$
has a wave solution of  the form $\psi=e^{kx+k^2y}\phi(Ux+z,t,k)$
such that the coefficients $\xi_s(z,y)$ of the formal series
\beq\label{psi2}
\phi(z,t,k)=e^{bt}\biggl(1+\sum_{s=1}^{\infty}\xi_s(z,t\, k^{-s}\biggr)
\eeq
are meromorphic functions of the variable $z\in \mathbb C^d$
with a simple pole at
the divisor $\Theta^U(t)$,
\beq\label{v1}
\xi_s(z+\l,t)=\xi_s(z,t)=\frac{\tau_s(z,t)}{\tau(z,t)}\, ;
\eeq

(ii) $\phi(z,t,k)$ is quasi-periodic with respect
to $\Lambda_U$, i.e., for
$\l\in \Lambda_U$
\beq\label{v10}
\phi(z+\l,t,k;z_0)=\phi(z,t,k;z_0)\,\mu_{\l}(k)
\eeq

(iii) $\phi(z,t,k)$ is unique up to a $\partial_U$-invariant factor which is an exponent of the linear form,
\beq\label{v2}
\phi_1(z,t,k)=\phi(z,t,k) e^{(\ell (k),z)}, \quad (\ell(k),U)=0.
\eeq
\end{lemma}

\subsection*{The spectral curve}

The next goal is to show that $\l$-periodic wave solutions of equation (\ref{intCM}) are common eigenfunctions of rings of commuting operators.

Note that a simple shift $z\to z+Z$, where $Z\notin \Sigma,$ gives
$\l$-periodic wave solutions with meromorphic coefficients along the affine
subspaces $Z+\mathbb C^d$. These $\lambda$-periodic wave solutions are related to each other
by $\p_U$-invariant factor. Therefore choosing, in the neighborhood of any
$Z\notin \Sigma,$ a hyperplane orthogonal to the vector $U$ and
fixing initial data on this hyperplane at $y=0,$ we define the corresponding
series $\phi(z+Z,t,k)$ as a \emph{local} meromorphic function of $Z$ and the
\emph{global} meromorphic function of $z$.

\begin{lemma} Let the assumptions of Theorem~\ref{th1.1} hold. Then there is a unique
pseudo\-differential operator
\beq\label{LL}
\L(Z,\p_x)=\p_x+\sum_{s=1}^{\infty} w_s(Z)\p_x^{-s}
\eeq
such that
\beq\label{kk}
\L(Ux+Vy+Z,\p_x)\,\psi=k\,\psi\,,
\eeq
where $\psi=e^{kx+k^2y} \phi(Ux+Z,t,k)$ is a $\l$-periodic solution of
(\ref{intCM}).
The coefficients $w_s(Z)$ of $\L$  are meromorphic functions on the abelian variety $X$
with poles along  the divisor $\Theta$.
\end{lemma}
\begin{proof}
Let $\psi$ be a $\l$-periodic wave solution. The substitution of (\ref{psi2}) in (\ref{kk})
gives a system of equations
that recursively define $w_s(Z,t)$ as differential polynomials in $\xi_s(Z,t)$.
The coefficients of $\psi$ are local meromorphic functions of $Z$, but
the coefficients of $\L$ are well-defined
\emph{global meromorphic functions} of on $\mathbb C^g\setminus\Sigma$, because
different $\l$-periodic wave solutions are related to each other by $\p_U$-invariant
factor, which does not affect $\L$. The singular locus is
of codimension $\geq 2$. Then Hartogs' holomorphic extension theorem implies that
$w_s(Z,t)$ can be extended to a global meromorphic function on $\mathbb C^g$.

The translational invariance of $u$ implies the translational invariance of
the $\l$-periodic wave solutions. Indeed, for any constant $s$ the series
$\phi(Vs+Z,t-s,k)$ and $\phi(Z,t,k)$ correspond to $\l$-periodic solutions
of the same equation. Therefore, they coincide up to a $\p_U$-invariant factor.
This factor does not affect $\L$. Hence, $w_s(Z,t)=w_s(Vt+Z)$.

The $\l$-periodic wave functions corresponding to $Z$ and
$Z+\lambda'$ for any $\lambda'\in \Lambda$
are also related to each other by a $\p_U$-invariant factor.
Hence, $w_s$ are periodic with respect to $\Lambda$ and therefore are
meromorphic functions on the abelian variety $X$.
The lemma is proved.
\end{proof}

\bigskip
Consider now the differential parts of the pseudo\-differential operators $\L^m$.
Let $\L^m_+$ be the differential operator such that
$\L^m_-=\L^m-\L^m_+=F_m\p^{-1}+O(\p^{-2})$. The leading
coefficient $F_m$ of $\L^m_-$ is the residue of $\L^m$:
\beq\label{res1}
F_m={res}_{\p}\  \L^m.
\eeq
 From the definition  of $\L$ it follows that $[\p_t-\p^2_x+u, \L^n]=0$. Hence,
\beq\label{lax}
[\p_t-\p_x^2+u,\L^m_+]=-[\p_t-\p_x^2+u, \L^m_-]=2\p_x F_m
\eeq

The functions $F_m$ are differential polynomials in the coefficients $w_s$ of $\L$.
Hence, $F_m(Z)$ are meromorphic functions on $X$. Next statement is crucial for
the proof of the existence of commuting differential operators associated with $u$.
\begin{lemma}[\cite{kr-schot}]The abelian functions $F_m$ have at most the second order pole on the divisor
$\Theta$.
\end{lemma}

Let $ {\hat F}$ be a linear space generated by $\{F_m, \ m=0,1,\ldots\}$, where we
set $F_0=1$. It is a subspace of the
$2^g$-dimensional space of the abelian functions that have at most second order pole at
$\Theta$. Therefore, for all but $\hat g=\dim\,{ \hat F}$ positive integers $n$,
there exist constants $c_{i,n}$ such that
\beq\label{f1}
F_n(Z)+\sum_{i=0}^{n-1} c_{i,n}F_i(Z)=0.
\eeq
Let $I$ denote the subset of integers $n$ for which there are no such constants. We call
this subset the gap sequence.
\begin{lemma} Let $\L$ be the pseudo\-differential operator corresponding to
a $\l$-periodic wave function $\psi$ constructed above. Then, for the differential operators
\beq\label{a2}
L_n=\L^n_++\sum_{i=0}^{n-1} c_{i,n}\L^{n-i}_+=0, \ n\notin I,
\eeq
the equations
\beq\label{lp}
L_n\,\psi=a_n(k)\,\psi, \ \ \ a_n(k)=k^n+\sum_{s=1}^{\infty}a_{s,n}k^{n-s}
\eeq
where $a_{s,n}$ are constants, hold.
\end{lemma}
\begin{proof} First note that from (\ref{lax}) it follows that
\beq\label{lax3}
[\p_t-\p_x^2+u,L_n]=0.
\eeq
Hence, if $\psi$ is a $\l$-periodic wave solution of (\ref{intCM})
corresponding to $Z\notin \Sigma$, then $L_n\psi$ is also a formal
solution of the same equation. That implies the equation
$L_n\psi=a_n(Z,k)\psi$, where $a$ is $\p_U$-invariant. The ambiguity in the definition of
$\psi$ does not affect $a_n$. Therefore, the coefficients of $a_n$ are well-defined
\emph{ global} meromorphic functions on $\mathbb C^g\setminus \Sigma$. The $\p_U$-
invariance of $a_n$ implies that $a_n$, as a function of $Z$, is holomorphic outside
of the locus. Hence it has an extension to a holomorphic function on $\mathbb C^g$.
Equations (\ref{v10}) imply that $a_n$ is periodic with respect to the lattice
$\Lambda$. Hence $a_n$ is $Z$-independent. Note that $a_{s,n}=c_{s,n},\ s\leq n$.
The lemma is proved.
\end{proof}

The operator $L_m$ can be regarded as a $Z\notin \Sigma$-parametric family   of
ordinary differential operators $L_m^Z$ whose coefficients have the form
\beq\label{lu}
L_m^Z=\p_x^n+\sum_{i=1}^m u_{i,m}(Ux+Z)\, \p_x^{m-i},\ \ m\notin I.
\eeq
\begin{corollary} The operators $L_m^Z$
commute with each other,
\beq\label{com1}
[L_n^Z,L_m^Z]=0, \ Z\notin \Sigma.
\eeq
\end{corollary}
From (\ref{lp}) it follows that $[L_n^Z,L_m^Z]\psi=0$. The commutator is an ordinary
differential operator. Hence, the last equation implies (\ref{com1}).

\begin{lemma}\label{lem6.7} Let $\A^Z,\ Z\notin \Sigma,$ be a commutative ring of ordinary differential
operators spanned by the operators $L_n^Z$. Then there is an irreducible algebraic
curve $\G$ of arithmetic genus $\hat g=\dim\, {\hat F}$ such that $\A^Z$ is isomorphic
to the ring $A(\G,P_0)$ of the meromorphic functions on $\G$ with the only pole at
a smooth point $P_0$. The correspondence $Z\to \A^Z$ defines a holomorphic
imbedding of $X\setminus \Sigma$ into the space of torsion-free rank 1 sheaves $\mathcal F$ on $\G$
\beq\label{is}
j\colon X\setminus\Sigma\longmapsto \overline{ Pic}(\G).
\eeq
\end{lemma}
The statement of the Lemma is a corollary of the following fundamental fact from the theory of commuting differential operators
\begin{theorem} \cite{ch,kr1,kr2,mum} There is a natural correspondence
\beq\label{corr}
\A\longleftrightarrow \{\G,P_0, [k^{-1}]_1, {\mathcal F} \}
\eeq
between \emph{regular} at $x=0$ commutative rings $\A$ of ordinary linear
differential operators containing a pair of monic operators of co-prime orders, and
sets of algebraic-geometrical data $\{\G,P_0, [k^{-1}]_1, \mathcal F \}$, where $\G$ is an
algebraic curve with a fixed
first jet $[k^{-1}]_1$ of a local coordinate $k^{-1}$ in the neighborhood of a smooth
point $P_0\in\G$ and $\mathcal F$ is a torsion-free rank 1 sheaf on $\G$ such that
\beq\label{sheaf}
H^0(\G,\mathcal F)=H^1(\G,\mathcal F)=0.
\eeq
The correspondence becomes one-to-one if the rings $\A$ are considered modulo conjugation
$\A'=g(x)\A g^{-1}(x)$.
\end{theorem}
Note that in \cite{kr1,kr2,ch} the main attention was paid to the generic case of
the commutative rings corresponding to smooth algebraic curves.
The invariant formulation of the correspondence given above is due to Mumford \cite{mum}.

The algebraic curve $\G$ is called the spectral curve of $\A$.
The ring $\A$ is isomorphic to the ring $A(\G,P_0)$ of meromorphic functions
on $\G$ with the only pole at the point $P_0$. The isomorphism is defined by
the equation
\beq\label{z2}
L_a\psi_0=a\psi_0, \ \ L_a\in \A, \ a\in A(\G,P_0).
\eeq

\begin{lemma}[\cite{kr-schot}] The linear space ${\hat F}$ generated by the abelian functions $\{F_0=1,
F_m=\res_\p \L^m\},$ is
a subspace of the space $H$ generated by $F_0$ and by the abelian functions
$H_i=\p_U\p_{z_i}\ln \theta(Z)$.
\end{lemma}

The construction of multivariable Baker-Akhiezer  functions presented for smooth curves
is a manifestation of general statement valid for singular spectral curves:
flows of the KP hierarchy define deformations of the commutative
rings $\A$ of ordinary linear differential operators. The spectral curve
is invariant under these flows. For a given spectral curve $\G$ the orbits of
the KP hierarchy are isomorphic to the generalized Jacobian $J(\G)={Pic}^0 (\G)$,
which is the equivalence classes of zero degree divisors on the spectral curve
(see details in \cite{shiota,kr1,kr2,wilson}).
Hence, for any $Z\notin \Sigma$, the orbit of the KP flows defines an holomorphic imbedding:
\beq\label{imb}
i_Z\colon J(\G)\longmapsto X.
\eeq
From (\ref{imb}) it follows that $J(\G)$ is \emph{compact}.

The generalized Jacobian of an algebraic curve is compact
if and only if the curve is \emph{smooth} (\cite{mdl}). On a smooth algebraic curve
a torsion-free rank 1 sheaf is a line bundle, i.e., $\overline { Pic} (\G)=J(\G)$.
Then (\ref{is}) implies that $i_Z$ is an isomorphism. Note that
for the Jacobians of smooth algebraic curves the bad locus $\Sigma$ is empty
(\cite{shiota}), i.e., the imbedding $j$ in (\ref{is}) is defined everywhere
on $X$ and is inverse to $i_Z$. Theorem~\ref{th1.1} is proved.

\section{Characterizing Jacobian of curves with involution}

As it was mentioned in Introduction the problem of characterization of Jacobians of curves with involution addressed in \cite{kr-inv} was motivated by construction of solutions of two-dimensional integrable systems with symmetries.
To the best of our knowledge from pure algebraic-geometri\-cal perspective the characterization problem of curves with involution in terms of their Jacobians has never been considered in its full generality. The only known to the author works in this direction are \cite{bv,grush,poor}.

\medskip
Two characterizations which distinguish such Jacobians were obtained in \cite{kr-inv} within the framework of  cases (i) and (ii) of Welter's conjecture. Both of them are limited to the case of involutions having at least one fixed point, i.e. to two-sheeted \emph{ramified} covers.

In a certain sense the setup we consider -- the Jacobian and the Prym variety in it -- resembles the setup arising in the famous Schottky-Yung relations and it is tempting to find a way to get these relations by means of the soliton theory. Unfortunately that challenging problem remains open.

The first characterization, related to the KP theory, is limited to the case of ramified cover by the obvious reason, since a curve with one marked point is used in constructing its solutions.

\begin{theorem}\label{thm:main} An indecomposable principally polarized abelian variety $(X,\theta)$ is the Jacobian variety of a smooth algebraic curve $\G$ of genus $g$ with involution $\s:\G\to \G$  having at least one point fixed if and only if there exist $g$-dimensional vectors $U\neq\,0,V,A,\zeta $ and constants $\Omega_1,\Omega_2, b_1$ such that:

the condition $(A)$ of Theorem \ref{th1.1} is satisfied and

$(B)$ the intersection of the theta-divisor $\Theta=\{Z\in X\,\mid\, \theta(Z)=0\}$ with a shifted abelian subvariety $Y\subset X$ which is the Zariski closure of $\pi(Ux+\zeta) \subset X$ is {reduced} and the equation
\beq\label{C}
\p_V\theta|_{\Theta\cap Y}=0
\eeq
holds.

Moreover, the locus $\Pi$ of points $\zeta\in X$ for which the equation \eqref{C} holds is the locus
of points for which the equation $\zeta+\sigma(\zeta)=2P+K\in X $, where $K$ is the canonical class,
holds.
\end{theorem}

The condition $(B)$ implies that

$(C)$ \emph{there is a constant $b_2$ such that the equality
\beq\label{b1}
\p_U\p_V\ln \theta|_{\widehat Y}=b_2
\eeq
 holds on $Y$}.

From the addition theorem \eqref{ThetaQuad} it follows that \eqref{b1} is equivalent to the condition that the vector $\left(\p_U\p_V K(0)-b_2K(0)\right)$ is \emph{orthogonal to the image under the Kummer map $K(\Pi)$ of the shifted abelian subvariety} $\widehat Y$:
\beq\label{ort}
\sum_{\e\in((1/2)\mathbb Z/\mathbb Z)^g}\left(\p_U\p_V\Theta[\e,0](0)-b_2\Theta[\e,0](0)\right)\Theta[\e,0](z)=0,\ \ z\in \widehat Y
\eeq
whence follows the condition \emph{of a kind of flatness} of the image under the Kummer map of the shifted Prym subvariety $\Pi \subset X $, that is, $ K (\Pi) $ lies in a proper (projective) linear subspace.

\medskip
The explicit meaning $(B)$ is as follows. As shown in \cite{fay,shiota} the affine line $Ux+Z$ is not contained in $\Theta$ for any vector $Z$. Hence, the function $\tau(x,t):=\theta(Ux+Vt+z), z\in Y$ is a \emph{nontrivial} entire function of $x$. The statement that $\Theta\cap Y$ is reduced means that zeros $q(t)$ of $\tau$ considered as a function of $x$ (depending on $t$) are generically simple, $\tau(q(t),t)=0,\, \tau_x(q(t),t)\neq 0$. Then \eqref{b1} is the equation
\beq\label{turningpoints}
\p_t\, q|_{t=0}=0
\eeq

\medskip
In the case, when $U$ spans an elliptic curve in the Jacobian, the statement that from $(B)$ it follows that the corresponding curve $\G$ admits an involution is obvious. Indeed, in that case the curve is the normalization of the spectral curve of $N$ point elliptic CM systems. The latter is defined by the characteristic equation
$$\det (k\cdot \mathbb I-L(z))=0$$

\noindent
of the matrix $L(z)$ defined in \eqref{LCM} with $q_i=q_i(0)$ and $p_i=\dot q_i(0)$, where $q_i(t)$  are roots of the equation  $\theta(Ux+Vt+z)=0$. If equation \eqref{turningpoints} holds, i.e. $p_i=0$, then it is easy to see that the matrix $L(z)$ satisfies the equation $L^t(z)=-L(-z)$. The latter implies that the curve is invariant under the involution $(k,z)\to (-k,-z)$. That observation made in \cite{kn} was the main motivation behind  \cite{kr-inv}.

At the heart of the proof in the general case is the statement that if $(B)$ is satisfied then there is a local coordinate $k^{-1}$ such that if $\psi(x,t,k)$ is the wave solution of \eqref{intCM} as in Lemma \ref{lem6.3} then $\psi(x,0,-k)=\psi^*(x,0,k)$ where $\psi^*$ is a wave solution of the equation
\beq\label{adjCM}
(\p_t+\p_x^2-u)\psi^*(x,t,k)=0
\eeq
which is formally adjoint to \eqref{intCM}.

\medskip
The second characterization of the Jacobians of curves with involution is related to the 2D Toda theory. A priory, unlike the KP case, there is no obvious reason why it is not applicable to all types of involution including unramified covers. It turned out that there is an obstacle for the case unramified covers and our second theorem also gives a characterization of the Jacobians of curves with involution \emph{with} fixed points.

\begin{theorem}\label{main2}
An indecomposable, principally polarized abelian variety $(X,\theta)$
is the Jacobian of a smooth curve of genus $g$ with involution having fixed points if and only if
there exist non-zero $g$-dimensional vectors
$U\neq A \, (\bmod\,  \Lambda),\, V, \zeta$, constants $\Omega_0,\Omega_1,b_1$ such that: the condition $(A)$  of Theorem \ref{th1.2} is satisfied and

$(B)$ $(i)$ The intersection of the theta-divisor with the shifted Abelian variety $Y$, which is a closure of $\pi(Ux+\zeta)$, is reduced and is not invariant under the shift by $U$, $\Theta\cap Y\neq (\Theta+U)\cap Y$, and $(ii) $ the equation
\beq\label{Cd}
\left((\p_V\theta(z))^2+\theta(z+U)\theta(z-U)\right)|_{z\in\Theta\cap Y}=0,
\eeq
holds.

Moreover, the locus of the points $\zeta\in X$ for which the equation \eqref{Cd} holds is the locus of point for which
the equation $\zeta+\zeta^\s=K+P_1+P_2\in J(\G)$, where $(P_1, P_2)$ are points of the curves permuted by $\s$ and such that $U=A(P_2)-A(P_1)$, is satisfied.
\end{theorem}
\begin{remark} {In the case when $U$ spans an elliptic curve in the Jacobian the statement of the theorem was proved first in \cite{kz2}}.
\end{remark}
The geometric form of the characterization is the condition that \emph{the vector $(2\p_V^2 K(0)-b_2K(U)-b_3K(0))$ is orthogonal to the image under the Kummer map of the abelian subvariety $\Pi$}:
\beq\label{ortd}
\sum_{\e\in((1/2)\mathbb Z/\mathbb Z)^g}(2\p^2_V\Theta[\e,0](0)-b_2\Theta[\e,0](U)-b_3\Theta[\e,0](0)))\Theta[\e,0](z)=0,
\eeq
where $z\in \Pi$ and $b_3$ is a constant.

\section{Non-local generating problem}

Until now our main focus was on equations that arise from the \emph{local} generating properties of two-dimensional linear operators with meromorphic coefficients. The \emph{non-local} generating properties of the same linear operators do not lead directly to equations of motion for zeros of the $\tau$ function. To begin with they generate the Lax representation of these equations. That non-local perspective is known for the elliptic case. Its abelian generalization is an open and challenging problem.

Let $\D$ be a linear differential or difference operator in
two variables $(x,t)$ with coefficients which are scalar or matrix elliptic
functions of the variable $x$ (i.e. meromorphic double-periodic functions
with the periods $2\omega_{\a}, \ \a=1,2$). We do not assume any special
dependence of the coefficients with respect to the second variable.
Then it is natural to introduce a notion of \emph{double-Bloch} solutions
of the equation
\beq
\D \Psi=0. \label{gen}
\eeq
We call a \emph{meromorphic} vector-function $f(x)$ that
satisfies the following monodromy properties:
\beq
f(x+2\omega_{\alpha})=B_{\alpha} f(x), \ \  \alpha=1,2,\label{g1}
\eeq
a \emph{double-Bloch function}.  The complex numbers $B_{\alpha}$ are  called
\emph{Bloch multipliers}.  (In other words, $f$ is a meromorphic section of a
vector bundle over the elliptic curve.)

In the most general form a problem considered in the framework of elliptic pole systems is to
\emph{classify} and to \emph{construct} all the operators $L$
such that equation (\ref{gen}) has \emph{sufficiently enough} double-Bloch
solutions.

It turns out that existence of the double-Bloch solutions is so
restrictive that only in exceptional cases such solutions do exist.
A simple and general explanation of that is due to the Riemann-Roch
theorem. Let $D$ be a set of points $q_i,\ i=1,\ldots,m,$ on the elliptic
curve $\G_0$ with multiplicities $d_i$ and let $V=V(D; B_1,B_2)$ be a linear
space of the double-Bloch functions with the Bloch multipliers
$B_\a$ that have poles at $q_i$ of order less or equal to $d_i$ and
holomorphic outside $D$.
Then the dimension of $D$ is equal to:
$$
{dim} \ D={deg} \ D=\sum_i d_i.
$$
Now let $q_i$ depend on the variable $t$. Then for $f\in D(t)$ the function
$\D f$ is a double-Bloch function with the same Bloch multipliers but in
general with higher orders of poles because taking derivatives and
multiplication by the elliptic coefficients increase orders.
Therefore, the operator $\D$ defines a linear operator
$$
\D|_D: V(D(t);B_1,B_2)\longmapsto V(D'(t);B_1,B_2), \ N'=\deg D'>N=\deg D,
$$
and (\ref{gen}) is \emph{always} equivalent to an \emph{over-determined}
linear system of $N'$ equations for $N$ unknown variables which are the
coefficients $c_i=c_i(t)$ of an expansion of $\Psi\in V(t)$ with respect to
a basis of functions $f_i(t)\in V(t)$. With some exaggeration one may say
that in the soliton theory the representation of a system  in the
form of the compatibility condition of an over-determined system of the
linear problems is considered as equivalent to integrability.

In all of known examples $N'=2N$ and
the over-determined system of equations has the form
\beq
LC=kC,\ \ \p_tC=MC,\label{laxmatrix}
\eeq
where $L$ and $M$ are $N\times N$ matrix functions depending on a point $z$
of the elliptic curve as on a parameter. A compatibility
condition of (\ref{laxmatrix}) has the standard Lax form $\p_t L=[M,L]$, and is
equivalent to a finite-dimensional integrable system.

The basis in the space of the double-Bloch functions can be written in terms
of the fundamental function $\Phi(x,z)$ defined by the formula \eqref{PhiCM}
Note, that $\Phi(x,z)$ is a solution of the Lame equation:
\beq
\Big(\frac{d^2}{dx^2}-2\wp(x)\Big)\Phi(x,z)=\wp(z)\Phi(x,z). \label{lame}
\eeq
From the monodromy properties it follows that $\Phi$ considered
as a function of $z$ is double-periodic:
$$
\Phi(x,z+2\omega_{\alpha})=\Phi(x,z) ,
$$
though it is not elliptic in the classical sense due to an
essential singularity at $z=0$ for $x\neq 0$.

As a function of $x$ the function $\Phi(x,z)$ is double-Bloch function, i.e.
$$
\Phi(x+2\omega_{\alpha}, z)=T_{\alpha}(z) \Phi (x, z), \ T_{\alpha}(z)=
\exp \left(2\omega_{\a}\zeta(z)-2\zeta (\omega _{\alpha})z\right).
$$
In the fundamental domain of the lattice defined by
$2\omega_{\alpha}$ the function $\Phi(x,z)$ has a unique pole at the point
$x=0$:
\beq
\Phi(x,z)=x^{-1}+O(x). \label{j}
\eeq
The gauge transformation
$$
f(x)\longmapsto \tilde f(x)=f(x)e^{ax},
$$
where $a$ is an arbitrary constant does not change poles of any function and
transform a double Bloch-function into a double-Bloch function. If $B_{\alpha}$
are Bloch multipliers for $f$ than the Bloch multipliers for $\tilde f$ are
equal to
\beq
\tilde B_1=B_1e^{2a\omega_1},\ \ \tilde B_2=B_2 e^{2a\omega_2}. \label{gn2}
\eeq
The two pairs of Bloch multipliers that are connected with each other
through the relation (\ref{gn2}) for some $a$ are called equivalent.
Note that for all equivalent pairs of Bloch multipliers the product
$
B_1^{\omega_2} B_2^{-\omega_1}
$
is a constant depending on the equivalence class, only.

From (\ref{j}) it follows that a double-Bloch function $f(x)$ with simple
poles $q_i$ in the fundamental domain and with Bloch multipliers
$B_{\alpha}$ (such that at least one of them is not equal to $1$) may be
represented in the form:
\beq
f(x)=\sum_{i=1}^N c_i\Phi(x-q_i,z) e^{k x},\label{g2}
\eeq
where $c_i$ is a residue of $f$ at $x_i$ and $z$, $k$ are parameters
related by
\beq
B_{\alpha}=T_{\alpha}(z) e^{2\omega_{\alpha}k }. \label{g3}
\eeq
(Any pair of Bloch multipliers may be represented in the form (\ref{g3})
with an appropriate choice of the parameters $z$ and $k$.)

To prove (\ref{g2}) it is enough to note that as a function of $x$
the difference of the left and right hand sides is holomorphic in the
fundamental domain.  It is a double-Bloch function with the same Bloch
multipliers as the function $f$. But a non-trivial double-Bloch function with
at least one of the Bloch multipliers that is not equal to $1$, has at least
one pole in the fundamental domain.

\medskip

\noindent
\subsubsection*{Example: Elliptic CM system.}

Let us consider the equation \eqref{intCM} with an elliptic (in $x$) potential $u(x,t)$. Suppose that equation \eqref{intCM} has $N$ linear independent double-Bloch solutions with equivalent
Bloch multipliers and $N$ simple poles $q_i(t)$. The assumption that there exist $N$ linear independent double-Bloch solutions with equivalent Bloch multipliers implies that they can be
written in the form
\beq
\Psi=\sum_{i=1}^N c_i(t,k,z)\Phi(x-q_i(t),z)e^{kx+k^2t}, \label{psi}
\eeq
with the same $z$ but different values of the parameter $k$.

Let us substitute (\ref{psi}) into (\ref{intCM}). Then
(\ref{intCM}) is satisfied if and if we get
a function holomorphic in the fundamental domain. First of all, we
conclude that $u$ has poles at $q_i$, only.
The vanishing of the triple poles $(x-q_i)^{-3}$ implies
that $u(x,t)$ has the form
\beq
u(x,t)=2\sum_{i=1}^N\wp(x-q_i(t))     \label{ufinal}
\eeq
The vanishing of the double
poles $(x-q_i)^{-2}$ gives the equalities that can be written as a matrix
equation for the vector $C=(c_i)$:
\beq (L(t,z)+k\mathbb I)C=0\,,
\label{l}
\eeq
where $I$ is the unit matrix and the Lax matrix $L(t,z)$ is defined in \eqref{LCM}.
Finally, the vanishing of the simple poles gives the equations
\beq (\partial_t-M(t,z))C=0\,,  \label{m} \eeq where \beq
M_{ij}=\left(\wp(z)-2\sum_{j\neq
i}\wp(q_i-q_j)\right)\delta_{ij}-2(1-\delta_{ij})\Phi' (q_i-q_j ,z). \label{M}
\eeq
The existence of $N$ linear independent
solutions for (\ref{intCM}) with equivalent Bloch multipliers implies that
(\ref{l}) and (\ref{m}) have $N$ independent solutions corresponding to
different values of $k$. Hence, as a compatibility condition we get the Lax
equation $\dot L=[M,L]$ for the elliptic CM system.


\begin{thebibliography}{99}


\bibitem{amkm} H. Airault, H. McKean, and J. Moser,
Rational and elliptic solutions of the Korteweg\,--\,de Vries equation and
related many-body problem. \emph{
Commun. Pure Appl. Math.}, {\textbf 30} (1977), no.\,1, 95--148.

\bibitem{akv}
A. Akhmetshin, I. Krichever, Yu. Volvoskii,
Elliptic families of solutions of the Kadomtsev-Petviashvili equation,
and the field analogue of the elliptic Calogero-Moser system.
\emph{Funct. Anal. Appl. }36 (2002), no. 4, 253--266

\bibitem{arb:expository}
E. Arbarello, Survey of Work on the Schottky Problem up to 1996.
\emph{Added section to the 2nd edition of Mumford's Red Book}, pp. 287--291, 301--304, Lecture Notes in Math. 1358, Springer, 1999.

\bibitem{arb-decon}
E. Arbarello, C. De~Concini, On a set of equations characterizing Riemann matrices.
\emph{Ann. of Math.} (2) 120 (1984), no. 1, 119--140.

\bibitem{arbarello}
E. Arbarello, C. De~Concini,
Another proof of a conjecture of S.P. Novikov on periods of
abelian integrals on Riemann surfaces. \emph{Duke Math. Journal}, 54 (1987), 163--178.

\bibitem{acp}
Enrico Arbarello, Giulio Codogni, Giuseppe Pareschi, Characterizing Jacobians via the KP equation and via flexes and degenerate trisecants to the Kummer variety: an algebro-geometric approach, arXiv:2009.14324

\bibitem{flex} E. Arbarello, I. Krichever, G. Marini,
Characterizing Jacobians via flexes of the Kummer Variety.
\emph{Math. Res. Lett.} 13 (2006), no. 1, 109--123.

\bibitem{krbab} O. Babelon, E. Billey, I. Krichever and M. Talon,
Spin generalisation of the Calogero\,--\,Moser system and the
matrix KP equation. \emph{ in ``Topics in Topology and Mathematical Physics'',
Amer. Math. Soc. Transl. Ser.\,2 {\textbf 170},
Amer. Math. Soc., Providence}, 1995, 83\,--119.



\bibitem{beauville}
A. Beauville, Le probl\`eme de Schottky et la conjecture de Novikov.
\emph{S\'eminaire Bourbaki, ann\'ee 1986--87, Expos\'e 675. Ast\'erisque }152--153
(1987), 101--112.

\bibitem{bv}
A.~Beauville, Vanishing thetanulls on curves with involutions.\emph{ Rend. Circ. Mat. Palermo }(2) 62 (2013), no. 1, 61–66.
\bibitem{bd}
A. Beauville, O. Debarre, Sur le probl\`eme de Schottky pour
les vari\'et\'es de Prym. \emph{Ann. Scuola Norm. Sup. Pisa}-- Cl. Sci., S\'er.~4,
{\textbf 14}, no 4 (1987) 613--623.

\bibitem{ch}
J.L. Burchnall, T.W. Chaundy, Commutative ordinary differential
operators.~I, II. \emph{Proc. London Math Soc.} {\textbf 21} (1922), 420--440 and
\emph{Proc. Royal Soc. London } {\textbf 118} (1928), 557--583.

\bibitem{DJKM81}
E. Date, M. Jimbo, M. Kashiwara and T. Miwa,
KP hierarchy of orthogonal and symplectic type -- Transformation
groups for soliton equations VI. \emph{J. Phys. Soc. Japan} {\textbf 50} (1981) 3813--3818.

\bibitem{DJKM83}
E. Date, M. Jimbo, M. Kashiwara and T. Miwa,
Transformation groups for soliton equations, \emph{in ``Nonlinear Integrable Systems --
Classical Theory and Quantum theory''}, M. Jimbo and T. Miwa (eds.), World Sci.,
Singapore, 1983, pp. 39--119.

\bibitem{deb}
O. Debarre, Vers une stratification de l'espace des modules des
vari\'et\'es ab\'eliennes principalement polaris\'ees. \emph{Complex
algebraic varieties (Bayreuth, 1990), 71--86, Lecture Notes in
Math. 1507, Springer, Berlin}, 1992.

\bibitem{deb:expository}
O. Debarre, The Schottky problem: an update. \emph{In: Current topics in
complex algebraic geometry }(Berkeley, CA, 1992/93); pp. 57--64.
(H. Clemens and J. Koll\'ar, eds.) \emph{MSRI Publ. 28, Cambridge Univ. Press,
Cambridge}, 1995.

\bibitem{mdl}
P. Deligne, D. Mumford, The irreducibility of the space of curves of given genus.
\emph{Inst. Hautes Etudes Sci. Publ. Math.} No. 36 1969 75--109.

\bibitem{grin}
A. Doliwa, P. Grinevich, M. Nieszporski, P. M. Santini,
Integrable lattices and their sub-lattices: from the discrete Moutard
(discrete Cauchy-Riemann) 4-point equation to the self-adjoint 5-point scheme.
arXiv:nlin/0410046.

\bibitem{donagi}
R. Donagi, Big Schottky. \emph{ Invent. Math.} 89 (1987), no. 3, 569--599.

\bibitem{donagi2}
R. Donagi, Non-Jacobians in the Schottky loci.
\emph{Annals of Math.} {\textbf 126} (1987), 193--217.

\bibitem{grush}
H.~ Farkas, S.~ Grushevsky, R.~ Salvati Manni, An explicit solution to the weak Schottky problem. \emph{Algebr. Geom.} 8 (2021), no. 3, 358–373.

\bibitem{far}
H.M. Farkas, H.E. Rauch, Period relations of Schottky type on Riemann surfaces.
\emph{Ann. of Math.} (2) 92 1970 434--461.

\bibitem{fay}
J.D. Fay, Theta functions on Riemann surfaces.
\emph{Lecture Notes in Math. 352. Springer-Verlag, Berlin-New York}, 1973.

\bibitem{fay2}
J.D. Fay, On the even-order vanishing of Jacobian theta
functions. \emph{ Duke Math. J.  }{\textbf 51} (1984) 1, 109--132.



\bibitem{geemen}
B. van Geemen, Siegel modular forms vanishing on the moduli space of curves.
\emph{Invent. Math.} 78 (1984), no. 2, 329--349.

\bibitem{kr-quad} S. Grushevsky, I. Krichever, Integrable discrete
Schr\"odinger equations and a characterization of Prym varieties by a pair of quadrisecants.
\emph{Duke Mathematical Journal}, 152 (2010), no 2, 318--371.

\bibitem{gun1}
R. Gunning, Some curves in abelian varieties. \emph{Invent. Math.} 66 (1982), no. 3, 377--389.

\bibitem{gun2}
R.C. Gunning, Some identities for abelian integrals.\emph{ Amer. J. Math. }108 (1986), no. 1,
pp. 39--74.

\bibitem{igusa}
J. Igusa, On the irreducibility of Schottky's divisor.\emph{
J. Fac. Sci. Univ. Tokyo Sect. IA Math.} 28 (1981), no. 3, 531--545 (1982).


\bibitem{kr1}
I.M. Krichever, Integration of non-linear equations by methods of algebraic geometry.
\emph{Funct. Anal. Appl.}, 11 (1977), no. 1, 12--26.

\bibitem{kr2}
I.M. Krichever, Methods of algebraic geometry in the theory of non-linear equations.
\emph{Russian Math. Surveys}, 32 (1977), no. 6, 185--213.

\bibitem{kr-dif}
I. Krichever, Algebraic curves and non-linear difference equation.
\emph{Uspekhi Mat. Nauk} 33 (1978), no. 4, 215--216.

\bibitem{kr-com}
I. Krichever, Commutative rings of ordinary linear differential
operators.\emph{ Funkts. Analiz i Ego Pril.}, 12 (3), 20--31 (1978)
[\emph{Funct. Anal. Appl.}, 12 (3) 175--185 (1978)].

\bibitem{kr3}
I.M. Krichever, Elliptic solutions of the Kadomtsev-Petviashvili
equation and integrable systems of particles. \emph{Funct. Anal. Appl.}, {\textbf 14}
(1980), n 4, 282--290.

\bibitem{krnest} I. Krichever, Elliptic solutions to difference non-linear equations and
nested Bethe ansatz equations, \emph{Calogero-Moser-Sutherland models}
(Montreal, QC, 1997), 249--271, CRM Ser. Math. Phys., Springer,
New York, 2000.
\bibitem{kr-schot}
I. Krichever, Integrable linear equations and the Riemann-Schottky problem.
In: \emph{Algebraic Geometry and Number Theory, Birkh\"auser, Boston}, 2006.


\bibitem{kr-prym}
I. Krichever, A characterization of Prym varieties.
\emph{Int. Math. Res. Not.} 2006, Art. ID 81476, 36 pp.

\bibitem{kr-tri}
I. Krichever, Characterizing Jacobians via trisecants of the
Kummer Variety. \emph{Ann. of Math.} {\textbf 172} (2010), 485--516.

\bibitem{kr-inv}
I.~Krichever, Characterizing Jacobians of algebraic curves with involution, arXiv:2109.13161

\bibitem{krwz} I. Krichever, O. Lipan, P. Wiegmann, and A. Zabrodin,
Quantum integrable models and discrete classical Hirota equations.
\emph{Comm. Math. Phys.}, {\textbf 188} (1997), no.\,2, 267--304.

\bibitem{kn}
I.Krichever, N.Nekrasov, Novikov-Veselov symmetries of the two dimensional O(N) sigma model, arXiv:2106.14201


\bibitem{n-kr}
I. Krichever, S. Novikov, Two-dimensional Toda lattice, commuting difference operators
and holomorphic vector bundles. \emph{Uspekhi Mat. Nauk} , {58} (2003) n 3, 51--88.

\bibitem{kp} I. Krichever and D.H. Phong, Symplectic forms in the theory of solitons.
\emph{Surveys in Differential Geometry}{\textbf IV}.
C.L. Terng and K. Uhlenbeck, eds.
pp. 239--313, International Press, 1998.

\bibitem{kr-shio}
I. Krichever, T. Shiota, Abelian solutions of the KP equation.
In: \emph{Geometry, Topology and Mathematical Physics.} V.M. Buchstaber
and I.M. Krichever, eds.
\emph{Amer. Math. Soc. Transl}. (2) 224, 2008, 173--191.

\bibitem{kr-shio1}
I. Krichever, T. Shiota, Abelian solutions of the soliton equations and geometry of abelian varieties.
In: \emph{Liaison, Schottky Problem and Invariant Theory.
M.E. Alonso, E. Arrondo, R. Mallavibarrena, I. Sols, eds.
Progress in Math}. vol. 280, Birkh\"auser, 2010, pp. 197--222.

\bibitem{kr-shiot}
I.~Krichever, T.~Shiota, Soliton equations and the Riemann-Schottky problem. \emph{Handbook of moduli}. Vol. II, 205–258, Adv. Lect. Math. (ALM), 25, Int. Press, Somerville, MA, 2013.h

\bibitem{kwz}
I. Krichever, P. Wiegmann, A. Zabrodin, Elliptic solutions to
difference non-linear equations and related many-body problems.
\emph{Comm. Math. Phys.}  193  (1998),  no. 2, 373--396.

\bibitem{zab} I.M. Krichever, A.V. Zabrodin, Spin generalization of
the Ruijsenaars-Schneider model, non-abelian 2D Toda chain and
representations of Sklyanin algebra. \emph{Uspekhi Mat. Nauk}, {\textbf 50}
(1995), no. 6 , 3--56.

\bibitem{kz1}

I.~Krichever, A.~ Zabrodin, Turning Points and CKP Hierarchy, \emph{Comm. Math. Phys.} 386 (2021), no. 3, 1643–1683.

\bibitem{kz2}
I. Krichever, A. Zabrodin, Constrained Toda hierarchy and turning points of the Ruijsenaars-Schneider model. arXiv:2109.05240


\bibitem{mar}
G.~Marini, A geometrical proof of Shiota's theorem on a conjecture of
S.P. Novikov. \emph{Compositio Math.} 111 (1998) 305--322.

\bibitem{mum:c+j}
D. Mumford, Curves and their Jacobians. \emph{University of Michigan Press},
Ann Arbor, 1975; also included in:
\emph{The Red Book of Varieties and Schemes, 2nd Edition.
Lecture Notes in Math.} 1358, Springer, 1999.

\bibitem{mum}
D. Mumford, An algebro-geometric construction of commuting operators and of solutions to
the Toda lattice  equation, Korteweg-de Vries equation and related non-linear equations. In:
\emph{Proceedings Int. Symp. Algebraic Geometry, Kyoto, 1977. M. Nagata, ed. 115--153, Kinokuniya Book Store,}
Tokyo, 1978.

\bibitem{mumford_prym} D. Mumford, Prym varieties I.
In: \emph{``Contributions to analysis''. L. Ahlfors, I. Kra, B. Maskit
and L. Nirenberg, eds. Academic Press}, 1974, pp 325--350.

\bibitem{poor}
C.~ Poor, The hyperelliptic locus. \emph{Duke Math. J.} 76 (1994), no. 3, 809–884.

\bibitem{wilson}
G. Segal, G. Wilson, Loop groups and equations of KdV type. \emph{IHES Publ. Math.} 61, 1985,
5--65.

\bibitem{schottky}
F. Schottky, Zur Theorie der Abelschen Functionen von vier Variabeln.
\emph{J. reine angew. Math.} {\textbf 102} (1888), 304--352.

\bibitem{schot-jung}
F. Schottky, H. Jung, Neue S\"atze \"uber Symmetrralfunktionen und die Abel'schen
Funktionen der Riemann'schen Theorie. \emph{S.-B. Preuss. Akad. Wiss. Berlin; Phys. Math. Kl}. 1 (1909) 282--297.



\bibitem{shiota}
T. Shiota, Characterization of {Jacobian} varieties in terms of soliton
equations. \emph{Invent. Math.}, 83(2), 333--382, 1986.


\bibitem{vdG}
G. van der Geer, The Schottky problem. In: \emph{Arbeitstagung Bonn 1984;
pp.~385--406. F. Hirzebruch et al., eds. Lecture Notes in Math}. 1111,
Springer, Berlin, 1985.

\bibitem{nv1}
A.P.~Veselov, S.P.~Novikov, Finite-zone, two-dimensional, potential Schr{\"o}dinger operators. Explicit formulas and evolution equations. \emph{Dokl.~Akad.~Nauk SSSR}, 279:1 (1984), 20-24\ .

\bibitem{nv2}
A.P.~Veselov, S.P ~Novikov, Finite-zone, two-dimensional Schr{\"o}dinger operators. Potential operators. \emph{Dokl.~Akad.~Nauk SSSR},  279:4 (1984), 784-788 \ .

\bibitem{wel}
G.E. Welters, On flexes of the Kummer variety (note on a theorem of R. C. Gunning).
\emph{Nederl. Akad. Wetensch. Indag. Math.} 45 (1983), no. 4, 501--520.

\bibitem{wel1}
G.E. Welters, A criterion for {Jacobi} varieties. \emph{Ann. of Math.},
120 (1984), no. 3, 497--504.








\end{thebibliography}
\end{document}